\documentclass[reqno,oneside]{amsart}

\usepackage[a4paper,text={14cm,21.7cm},centering]{geometry}

\usepackage[utf8]{inputenc}
\usepackage[T1]{fontenc}
\usepackage[svgnames,hyperref]{xcolor}
\usepackage{lmodern,amsmath,amssymb}

\DeclareFontFamily{U}{mathb}{\hyphenchar\font45}
\DeclareFontShape{U}{mathb}{m}{n}{
      <5> <6> <7> <8> <9> <10> gen * mathb
      <10.95> mathb10 <12> <14.4> <17.28> <20.74> <24.88> mathb12
      }{}
\DeclareSymbolFont{mathb}{U}{mathb}{m}{n}
\DeclareMathSymbol{\bigast}{1}{mathb}{"06}

\usepackage{tikz}
\usetikzlibrary{calc,positioning}
\usepackage{tikz-cd}
\tikzset{bullet/.style={circle,fill=black,minimum size=4.5pt,inner sep=0pt}}

\usepackage[shortlabels]{enumitem}
\setlist[1]{wide}
\setlist[enumerate]{label=\rm{(\arabic*)}}
\setlist[enumerate,2]{label=\rm({\it\roman*})}
\setlist[itemize]{label=\raisebox{0.25ex}{\tiny$\bullet$}}

\theoremstyle{plain}
 
\newtheorem{theorem}{Theorem}
\newtheorem{corollary}[theorem]{Corollary}
\newtheorem{proposition}[theorem]{Proposition}
\newtheorem{lemma}[theorem]{Lemma}
\newtheorem{fact}[theorem]{Fact}

\theoremstyle{definition}
 
\newtheorem{definition}[theorem]{Definition}

\newtheorem{example}[theorem]{Example}
\newtheorem{remark}[theorem]{Remark}
\newtheorem{setup}[theorem]{Set-Up}

\newcommand{\C}{\mathbf{C}}

\newcommand{\Q}{\mathbf{Q}}
\newcommand{\Z}{\mathbf{Z}}
\newcommand{\N}{\mathbf{N}}
\newcommand{\K}{\mathbf{k}}

\renewcommand{\P}{\mathbb{P}}
\renewcommand{\H}{\mathbb{H}}

\newcommand{\Cl}{\mathcal{C}}

\newcommand{\Ql}{\mathcal{Q}}

\DeclareMathOperator{\GL}{GL}

\DeclareMathOperator{\Aut}{Aut}

\DeclareMathOperator{\Min}{Min}
\DeclareMathOperator{\CAT}{CAT}
\DeclareMathOperator{\Isom}{Isom}

\DeclareMathOperator{\car}{char}

\DeclareMathOperator{\Fix}{Fix}
\DeclareMathOperator{\dist}{d}
\DeclareMathOperator{\axe}{ax}
\DeclareMathOperator{\diam}{diam}

\DeclareMathOperator{\Sat}{Sat}

\DeclareMathOperator{\Adj}{Adj}

\renewcommand{\le}{\leqslant}
\renewcommand{\leq}{\leqslant}
\renewcommand{\ge}{\geqslant}
\renewcommand{\geq}{\geqslant}
\newcommand{\eps}{{\varepsilon}}
\renewcommand{\epsilon}{{\varepsilon}}
\renewcommand{\phi}{\varphi}
\newcommand{\id}{\text{\rm id}}
\renewcommand{\setminus}{\smallsetminus}

\newcommand{\cone}{\mathring}
\newcommand{\action}{\curvearrowright}
\newcommand{\lX}{\ell_X}

\newcommand{\llb}{\langle\!\langle}
\newcommand{\rrb}{\rangle\!\rangle}

\newcommand{\mycolor}{Navy}
\usepackage[pdfauthor={S. Lamy \& A. Lonjou}, ocgcolorlinks, colorlinks, linktocpage, citecolor = \mycolor,
linkcolor = \mycolor, urlcolor = \mycolor]{hyperref}
\usepackage[all]{hypcap} 

\renewcommand{\l}{\ell} 

\title{Introduction to a small cancellation theorem }
\author{Stéphane Lamy}
\author{Anne Lonjou}
\date{\today}
\address{Institut de Mathématiques de Toulouse UMR 5219, Université de Toulouse, UPS F-31062 Toulouse Cedex 9, France}
\email{slamy@math.univ-toulouse.fr}
\address{Departement Mathematik und Informatik, Universität Basel, Spiegelgasse 1, 4051 Basel, Switzerland}
\email{anne.lonjou@unibas.ch}

\begin{document}

\begin{abstract}
This note is intended as an introduction to two previous works respectively by Dahmani, Guirardel, Osin, and by Cantat, Lamy.
We give two proofs of a Small Cancellation Theorem for groups acting on a simplicial tree.
We discuss the application to the group of plane polynomial automorphisms over any ground field. 
\end{abstract}

%
%

\maketitle


\section{Introduction}

The classical setting of Small Cancellation Theory \cite{LyndonSchupp} is the study of finitely presented groups $G= \langle S \mid R \rangle$.
Here $S = \{a_1, \dots, a_s\}$ is a finite set, and $R$ is a finite collection of cyclically reduced words called relators, stable under taking inverse and under cyclic permutations.
Then $G$ is the quotient of the free group $F_S$ over $S$ by the normal subgroup $\llb R \rrb$ generated by the relators.
The small cancellation condition asks for a uniform bound on overlaps between relators: there exists $\lambda > 0$ such that for any $r_1 \neq r_2$ in $R$, the length of a common prefix is at most $\lambda$ times the length of the relators.
When $\lambda$ is small such a condition ensures several algebraic properties of the quotient, the most basic one being that the quotient is non trivial, or equivalently that the normal subgroup $\llb R \rrb$ is a proper subgroup of $F_S$.

The proofs are geometrical and relies on proving a ``conservation of negative curvature'': starting from the action of the free group $F_S$ on its Cayley graph view as a metric tree, one associates to each relation a planar diagram whose regions correspond to elements in $R$.
The celebrated condition $C'(1/6)$ that asks for overlaps of at most $1/6$ of the length of the relators, is related to the fact that if each region in a planar diagram admits at least 6 (resp. 7) neighbors then the discrete Gauss--Bonnet formula ensures a curvature $\le 0$ (resp. $<0$), which in turns implies that the boundary perimeter grows at least linearly (resp. exponentially) with the number of regions.

Recall that a geodesic metric space is Gromov hyperbolic if all its geodesic triangles are $\delta$-thin for a uniform constant $\delta \ge 0$. 
Following the landmark paper \cite{Gromov}, the Small Cancellation Theory was first extended to the situation of hyperbolic groups \cite{Champetier}, \cite{Olshanskii}, \cite{Delzant}, which are finitely presented groups whose Cayley graph is Gromov hyperbolic.

Then simultaneously around 2010 two versions of a Small Cancellation Theorem appeared, with applications to distinct groups in mind:
\cite{DGO} with groups from topology, such as mapping class groups, and \cite{CL} with groups from algebraic geometry, such as Cremona groups.
For these applications one needs to consider group actions $G \action X$, where $X$ is a Gromov hyperbolic space that does not have to be a Cayley graph, $G$ is a group not necessarily finitely generated, and the action does not have to be proper.
Infinite stabilizers must be allowed, and the metric space $X$ does not have to be locally compact.
The statements are of the form: given such an action $G \action X$, if $g$ satisfies some variant of the small cancellation condition, then the normal subgroup generated by a sufficiently high power of $g$ is free and proper.

The precise small cancellation condition can take several forms, were one wants a condition as simple as possible to check in practice, and with strong algebraic consequences.
One such condition is the WPD (Weakly Proper Discontinuous) condition on a loxodromic isometry $g \in G$, which asks for some properness of the group action along the axis of $g$ (see \S\ref{sec:WPD} for the precise definition).
We can now give a precise statement:

\begin{theorem}[Small Cancellation Theorem]
\label{thm:main}
Let $G \action X$ be a group acting on a Gromov hyperbolic space.
Let $g \in G$ be a WPD loxodromic element.
Then for any $C >0$, there exists an integer $n_0$ such that for any $n \ge n_0$, the normal subgroup  $\llb g^n \rrb$ generated by $g^n$ in $G$ satisfies:
\begin{enumerate}
\item There is a collection $S$ of conjugates of $g^n$ such that  $\llb g^n \rrb$ is the free group over the elements in $S$.
\item Any element $h \neq \id$ in $\llb g^n \rrb$ is a loxodromic isometry with translation length $\ell(h) > C$.
\item In particular if $C > \ell(g)$ then $\llb g^n \rrb$ does not contain $g$, and so is a proper normal subgroup in $G$.
\end{enumerate}
\end{theorem}

This statement was used by the second author to get a proof of the non-simplicity of the Cremona group over any field \cite{Lonjou}.
The proof of Theorem \ref{thm:main} is quite technical in both above mentioned papers.
In \cite[Theorem 5.3 \&  Proposition 6.34]{DGO}, the proof relies on the techniques of rotating families and cone-off of a hyperbolic space.
In \cite[Theorem 2.10]{CL} the starting condition is the notion of tight isometry, which is similar to the WPD condition, and the proof relies on approximations by trees.

The archetypal example of a Gromov hyperbolic space being a tree, one can expect that a proof of Theorem \ref{thm:main} in the case of a group acting on a tree will be technically simpler while still retaining most of the flavor of the general case.
Our aim in this introductory note is to give two full proofs of a version of Theorem \ref{thm:main}, for groups acting on simplicial trees:

\begin{theorem}[Small Cancellation Theorem for a simplicial tree]
\label{thm:SCT}
Let $G \action X$ be a group acting on a simplicial tree, and $g \in G$ a loxodromic WPD element.
Then for any sufficiently large $n \in \N$, the normal subgroup  $\llb g^n \rrb$ generated by $g^n$ in $G$ satisfies:
\begin{enumerate}
\item\label{SCT1} 
There is a collection $S$ of conjugates of $g^n$ such that  $\llb g^n \rrb$ is the free group over the elements in $S$.
\item\label{SCT2}  
Any element $h \neq \id$ in $\llb g^n \rrb$ either is conjugate to $g^n$, or is a loxodromic isometry with translation length $\ell(h) > \ell(g^n)$.
\item\label{SCT3} 
In particular $\llb g^n \rrb$ does not contain $g$, and is a proper normal subgroup in $G$.
\end{enumerate}
\end{theorem}

In \S\ref{sec:DGO} we give a proof of Theorem \ref{thm:SCT} along the lines of \cite{DGO}, and in \S\ref{sec:CL} we give another proof following \cite{CL}.
Then we compare the notion of small cancellation, WPD and tight, and following \cite{MO} we illustrate the use of Theorem \ref{thm:SCT} by considering the action of the plane polynomial automorphism group on its Bass-Serre tree.
This is a natural subgroup of the Cremona group where we have an action on a non locally finite simplicial tree.

Observe that with the application to the Cremona group in mind, it would be sufficient to prove Theorem \ref{thm:main} when $X$ is a $\CAT(-1)$ space (or in fact simply the hyperbolic space $\H^n$, with $n = \infty$).
However it is unclear to us whether either proof can be much simplified with this extra assumption.
Another remark is that the proof in \cite{CL} does not seem to be able to give easily the free group structure, or other consequences such as the SQ-universality.

For most of the definitions and statement of the paper we indicate a reference, usually either to \cite{DGO} or to \cite{CL}.
These should be understood as ``compare our special definition in the case of a tree with the general definition of the cited paper''. 
For the application to polynomial automorphisms we will recall some results from \cite{MO}.
The proof in section~\ref{sec:CL} is a reworking of a proof that was included in the Habilitation thesis of the first author.

\section{Preliminaries}

\subsection{Trees and Bass-Serre theory}

A \emph{geodesic segment} between two points $x,y$ of a metric space $X$ is a map $\gamma\colon [a,b] \to X$ from a real interval, such that $\gamma(a) = x$, $\gamma(b) = y$, and $\gamma$ is an isometry onto its image.
We will denote $[x,y]$ such a segment, committing two abuses of notation: we identify the map with its image, and the geodesic segment is not unique in general.

A metric space $X$ is \emph{geodesic} if there exists at least one geodesic segment between each pair of points of $X$.    
A \emph{triangle} in a geodesic metric space is a choice of three points $x,y,z$, and a choice of three geodesic segments $[x,y]$, $[y,z]$, $[z,x]$  between them.
A \emph{tripod} in a geodesic metric space is the union of three geodesic segments whose pairwise intersections are reduced to a common endpoint $p$.
The point $p$ is called the \emph{branch point} of the tripod.
A \emph{real tree} is a geodesic metric space where all triangles are tripods.

Given an isometry $h$ of a metric space $X$, we define the \emph{translation length} of $h$ as the infimum $\l(h) = \inf_{x \in X} \dist(x, hx)$. 
We denote the set of points realizing the translation length by $\Min(h)$.
In the case where $X$ is a tree, we recall the following property: for any isometry $h$, and any point $x \in X$, the middle point of the segment $[x,hx]$ belongs to $\Min(h)$.
An isometry $h$ of a real tree $X$ is \emph{elliptic} if it admits at least one fixed point: in this case $\Min(h)$ is the subtree of fixed points of $h$. 
If an isometry of a real tree is not elliptic then it has to be \emph{loxodromic}: the translation length $\l(h)$ is positive, and the set of points realizing the infimum is a geodesic line $\axe(h) = \Min(h)$ called the axis of $h$.

A \emph{simplicial tree} is the topological realization of a combinatorial tree, where we always assume it is endowed with the metric such that each edge is isometric to the segment $[0,1]$.
We will need the following basic fact from Bass-Serre theory, which describes the structure of groups acting on simplicial trees.

\begin{proposition}\label{pro:bass serre}
Let $G \action X$ be a group acting on a simplicial tree, with trivial stabilizers of edges.
Denote by $G_v$ the stabilizer of a vertex $v$.
Assume that the quotient $X/G$ is a tree.
Then there exists a subtree $X' \subset X$ which is a fundamental domain, and $G$ is the free product of the $G_v$, where $v$ runs over the vertices of $X'$.
\end{proposition}

\begin{proof}
The existence of the fundamental domain $X'$ is \cite[Proposition 17 p. 32]{Serre}.
Then the free product structure follows from \cite[Theorem 10 p. 39]{Serre}.
\end{proof}

\subsection{Small cancellation conditions}
\label{sec:WPD}

We shall use the following definition in the context of a loxodromic isometry of a tree, but it is not harder to state it for a general metric space.

\begin{definition}[WPD isometry, {\cite[Definition 6.1]{DGO}, \cite{BF}}]
\label{def:WPD}
Let $G \action X$ be a group acting on a metric space. 
An element $g \in G$ satisfies the WPD condition (or is a \emph{WPD element}) if for some (hence for every) $x \in X$, and for every $\eps \ge 0$, there exists $N = N(\eps)$ such that the set
\[
\Fix_\eps \{x,g^N x\} = \left\lbrace h \in G \mid \dist(x,hx) \le \eps \text{ and } \dist(g^N x,hg^Nx) \le \eps \right\rbrace
\]
is finite.
\end{definition}

The fact that we can equivalently put ``for some'' or ``for every'' in the definition is an exercise, see also \cite[Lemma 1.1]{Lonjou}.
Observe that $\Fix_0 \{x,g^N x\}$ is a subgroup, but in general for $\eps > 0$ the set $\Fix_\eps \{x,g^N x\}$ is only stable under taking inverse, not under composition.
In the sequel we will write simply $\Fix$ instead of $\Fix_0$. 

\begin{definition}[Tight isometry, {\cite[\S 2.3.3]{CL}}]
\label{def:tight}
Let $G \action X$ be a group acting on a tree.
A loxodromic isometry $g \in G$ is \emph{tight} if there exists $B > 0$ such that for all $\varphi\in G$, the condition ${\rm diam}(\axe(g) \cap \varphi(\axe(g)) > B$ implies  $\varphi g \varphi^{-1}= g$ or $ g^{-1}$ (so that in particular $\varphi(\axe(g)) = \axe(g)$).
\end{definition}

Observe that if $g$ is tight, then any iterate $g^n$, $n > 0$, is also tight.

\begin{definition}[Small cancellation, {\cite[Definition 6.22]{DGO}}]
\label{def:SC}
Let $G \action X$ be a group acting on a tree. 
Let $g \in G$ be a loxodromic isometry with translation length $\ell(g)$, and $g^G$ its conjugacy class.
We say that $g^G$ satisfies the $\eps$-small cancellation condition if for any $h_1 g h_1^{-1} \neq h_2 g^{\pm 1} h_2^{-1}$ in $g^G$, we have 
\[
\eps \ell(g) > \diam \left( \axe(h_1 g h_1^{-1}) \cap \axe(h_2 g h_2^{-1}) \right).\] 
\end{definition}

\begin{proposition}
\label{pro:WPD=>SC}
Let $G \action X$ be a group acting on a simplicial tree, and $g \in G$ a loxodromic element.
\begin{enumerate}
\item If $g$ is a WPD element, then there exists $M \in \N^*$ such that $g^M$ is tight. 

\item If $g$ is tight, then for any $\eps >0$ there exists $n \in \N^*$ such that the conjugacy class
$\{g^n\}^G$ satisfies the $\eps$-small cancellation condition.
\end{enumerate}
\end{proposition}

\begin{proof}
\begin{enumerate}
\item 
We want to prove that there exist integers $B,M$ such that for any $\varphi\in G$, if 
\[
\diam \left(\axe(g)\cap \axe(\phi g \phi^{-1})\right) > B,
\]
then $g^M = \phi g^{\pm M} \phi^{-1}$.

Since $g$ is WPD, there exists an integer $p$ such that for any vertex $x \in
\axe(g)$, $\Fix \{x, g^p x\}$ is finite.
So there is $N$ such that $|\Fix \{x, g^p x\}| < N$ independently of the choice
of the vertex $x$ (here we use the fact that the tree is simplicial). 
Let $g' = \phi g \phi^{-1}$ such that 
\[D=\diam \left(\axe(g)\cap \axe(g') \right)> 2\l(g),\]
and pick a base point $x \in \axe(g)$.
Up to replacing $g'$ by a conjugate $g^q g' g^{-q}$, we can assume that $x$ is in the segment $\axe(g) \cap \axe (g')$, but that $g^{-1} x$ is not.
Up to replacing $g'$ by its inverse we can assume $g'^{-1}g$ fixes pointwise a segment $I \in \axe(g) \cap \axe(\phi g \phi^{-1})$ of diameter $D - \ell(g)$.
Similarly, for any integer $i$ such that $D - i\ell(g)\ge 0$, the element $g'^{-i}g^i $ fixes pointwise a segment of diameter $D - i\ell(g)$. 
Choose $B$ such that $B - N\ell(g) \ge (p+1)\ell(g)$. Then $ g'^{-i}g^i  \in \Fix \{x, g^p x\}$ for any $0\leq i\leq N$. So there exist $0 \leq j < i \leq N$ such that $g'^{-i}g^i  =  g'^{-j}g^j$, so that $g^{i-j} = g'^{i-j}$ with $0 < i-j \le N$.
Then we can take $M = N!$, or any multiple.

\item
It suffices to take $n$ such that $n\ell(g) \eps > B$. \qedhere
\end{enumerate}
\end{proof}

\begin{remark}
In general, without taking a power there is no direct implication between the notion of tight and WPD element (see Examples \ref{ex:WPDnotTight} and  \ref{ex:TightNotWPD}). 
\end{remark}

\begin{lemma}{\cite[Lemma 4.2 and Corollary 4.3]{MO}} \label{lem:cor4.2}
Let $G \action X$ a group acting on a simplicial tree, and $g \in G$ loxodromic.
Assume that there exist $u,v \in \axe(g)$ such that $\Fix \{u,v\}$ is finite.
Then $g$ is a WPD element.
\end{lemma}

\begin{proof}
Let $\eps \ge 0$, and $m \in \N$ such that $\dist(u,g^mu) > \max \{ \dist(u,v),\eps \}$.
We will show that $\Fix_\eps \{u,g^{3m}u \}$ is finite, so that $g$ is a WPD element.
Setting $x= g^{-m}u$ and $y = g^{2m}u$,  this is equivalent to showing that $\Fix_\eps \{x,y \}$ is finite.
Up to replacing $g$ by its inverse, we can assume that $v \in [u,g^m u]$.

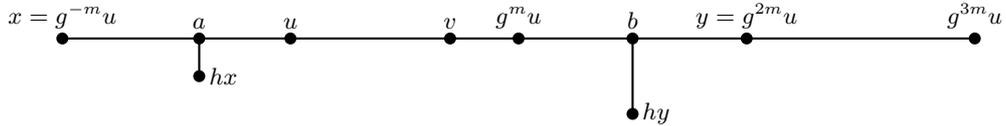
\[
\begin{tikzpicture}[node distance=3cm,font=\small,thick]
\coordinate[label=above:{$x = g^{-m}u$}] (x) at (0,0);
\coordinate [right=of x, label=above:$u$] (u); 
\coordinate [right=of u, label=above:$g^m u$] (gu);
\coordinate [right=of gu, label=above:{$y = g^{2m} u$}] (y);
\coordinate [right=of y, label=above:{$g^{3m} u$}] (gy);
\coordinate [label=above:$a$] (a) at ($ (x)!.6!(u) $) {};
\coordinate [below=of a,yshift=2.5cm,label=right:$h x$] (hx) {};
\coordinate [label=above:$b$] (b) at ($ (gu)!.5!(y) $) {};
\coordinate [below=of b,yshift=2cm,label=right:$h y$] (hy) {};
\coordinate [label=above:$v$] (v) at ($ (u)!.7!(gu) $) {};
\draw (x)--(gy) (a)--(hx) (b)--(hy);
\foreach \v in {x,u,gu,y,gy,v,a,hx,b,hy} {
  \node[bullet] at (\v) {};
}
\end{tikzpicture}
\]

Consider $h \in \Fix_\eps \{x,y \}$, so that we also have $h^{-1} \in \Fix_\eps \{x,y \}$.
Let $a, b$ be the respective projections of $hx, hy$ on the segment $[x,y]$ (the figure illustrates the most difficult case where $a$ and $b$ lie in the interior of this segment).
Since $\dist(x,hx) \le \eps < \dist(u,g^m u) = \dist(x,u)$, we have $a \in [x, u]$, and similarly $b \in [g^m u, y]$.
In particular $[u,v] \subset [a,b] = [x,y] \cap [hx,hy]$.
Similarly, working with $h^{-1}$ instead of $h$ we get $[u,v] \subset [x,y] \cap [h^{-1}x,h^{-1}y]$, and translating by $h$ this gives $h[u,v] \subset [a,b] \subset [x,y]$. 
Moreover since the set $\Min(h)$ contains the respective middle points of $[x, hx]$ and $[y, hy]$, we get that $[u,v] \subset \Min (h)$, so the segments $[u,v]$ and $[hu, hv]$ have the same orientation inside $[x,y]$.   

Since $X$ is simplicial, there exist finitely many $t_i \in G$, $i = 1, \ldots, r$, such that for any $f \in G$, if $[fu, fv] \subset [x,y]$ with the same orientation as $[u,v]$, then $fu = t_iu$ and $fv = t_iv$ for some $i$. 
In particular,  $t_i^{-1}h \in \Fix \{u,v\}$ for some $i$.
By assumption, $\Fix\{u,v\} = \{f_j;\; j = 1, \dots, s\}$ for some $f_j \in G$.
Finally 
\[\Fix_\eps \{x,y \} \subset \{t_if_j \mid 1 \le i \le r, 1 \le j \le s\}\]
is finite, as expected.
\end{proof}

\section{Group acting on a tree, following \cite{DGO}}
\label{sec:DGO}

We want to prove Theorem \ref{thm:SCT}, following the main lines of \cite{DGO}.
First we introduce the notion of cone-off of the tree with respect to a family of geodesic, adapting \cite[\S5.3]{DGO}.
Then we define the key notion of windmill and explain how to enlarge windmills in order to prove the main theorem.
In the last subsection we discuss how much we departed from the definitions in \cite{DGO}, in order to take advantage of the fact that the group acts on a simplicial tree.

\subsection{Cone-off}

\begin{setup} \label{setup:family}
Let $X$ be a simplicial tree, $g \in \Isom(X)$ a WPD loxodromic element.
We denote by $\Ql$ the collection of axes of conjugates of $g$ (or equivalently of $g^n$, for any integer $n\neq 0$):
\[
\Ql = \{ \axe(\phi g \phi^{-1}) \mid \phi \in G\}.
\]
We define:
\[ \Delta = \underset{Q_1\neq Q_2 \in \Ql}{\sup}\diam(Q_1\cap Q_2).\]
By Proposition \ref{pro:WPD=>SC} we have $\Delta < \infty$. 
We choose $n\in \N$ large enough such that 
\[ \l_X(g^n)>7\Delta. \] 

Given these data we make the following definitions:
\begin{itemize}
\item
The \textit{cone-off} $\cone X$ (of $X$ relatively to the
family $\Ql$) is defined as the graph obtained from $X$ by adding one new vertex $c_Q$ for each geodesic $Q \in \Ql$, and by putting an edge of length $r_0 > 0$ between $c_Q$ and each vertex of $Q$.
We choose once and for all the radius $r_0$ to be a real number such that $0 < 2r_0 < 1$, which justifies that we keep the radius implicit in the notation.
\item 
There is a natural injection from $X$ to $\cone X$, which allows us to see $X$ as a subset of $\cone X$.
Since $\Ql$ is preserved by $G$, the action of $G$ on $X$ naturally extends as an action by isometries on $\cone X$.
In the following, when we speak of segments, diameters, etc, by default we mean relatively to the distance in the cone-off $\cone X$. When we need to consider the same notions relatively to the distance in the tree $X$, we shall use a subscript such as $[x,y]_X$, $\diam_X$ or $\lX(g^n)$. 
\item 
We call each new vertex $c_Q$ an \emph{apex}, and we denote by $\Cl$ the family of apices in $\cone X$.
\item 
For each $Q \in \Ql$, we denote $\cone Q \subset \cone X$ the union of $Q$ with all edges to the apex $c_Q$.
\item 
Let $A$ be a subgraph of $\cone{X}$.
We say that $Q\in \Ql$ is \textit{adjacent to} $A$ if $Q$ intersects $A$ but
$c_Q$ does not belong to $A$.
In that case we also say that $c_Q$ is an adjacent vertex to $A$, and we denote $\Adj(A)$ the set of all adjacent vertices to $A$.
\item
We say that an edge $e \subset X$ is an \emph{insulator edge} if $e$ is not contained in any $Q \in \Ql$.
\item
If $c$ is the apex associated with $Q = \phi g \phi^{-1}$, we will use interchangeably the notation $G_c$ or $G_Q$ to denote the subgroup $\langle \phi g^n \phi^{-1} \rangle$. 
\end{itemize}
\end{setup}

We describe now geodesics in $\mathring{X}$ between two points of the tree. 
Then we give a sufficient condition forcing them to pass through a given apex.

\begin{proposition} \label{pro:geodesic}
Let $x,y \in X$ be two vertices.
\begin{enumerate}
\item \label{geodesic1} 
The distance $\dist(x,y)$ in $\cone X$ is equal to $m' + 2mr_0$, where $m'$ is the number of insulator edges in $[x,y]_X$, and $m$ is the minimal number of geodesics $Q_1, \cdots$, $Q_m \in \Ql$ necessary to cover all non insulator edges of $[x,y]_X$.
\item \label{geodesic2} 
Given any such minimal covering collection of geodesics $Q_1, \cdots, Q_m$, with associated apices $c_1, \cdots, c_m$, we can write $[x,y]$ as the concatenation of the insulator edges and of segments of the form $[a_i, c_i] \cup [c_i, b_i]$, with $[a_i, b_i]_X \subset Q_i$.    
\end{enumerate}
\end{proposition}

\begin{proof}
Let $[x,y] \subset \cone X$ be a geodesic segment.
Observe that if $e \subset [x,y]$ is an edge of $X$, it must be an insulator edge, otherwise replacing this edge by a shortcut of length $2r_0 < 1$ passing through an apex we would contradict the assumption that $[x,y]$ is geodesic.  
Let $c_1, \dots, c_k$ be the apices contained in $[x,y]$, ordered such that $\dist(x,c_i) < \dist(x,c_{i+1})$ for each $i$. 
Denote by $Q_i$ the geodesic corresponding to the apex $c_i$. 
Let $[a_i,c_i] \cup [c_i, b_i]$ be the segment of length $2r_0$ that is obtained as the intersection of $[x,y]$ with the cone $\cone Q_i$.
Then the union of the insulator edges and of the paths $[a_i, b_i]_X$ form a connected subtree of $X$ containing $x$ and $y$, and so also $[x,y]_X$. 
This shows that all $m'$ insulator edges in $[x,y]_X$ must be in $[x,y]$, and $[x,y]$ is of length $m' + 2kr_0$.

Conversely, given a minimal collection $Q_1, \cdots, Q_m \in \Ql$ covering the non insulator edges of $[x,y]_X$, we define inductively a subsegment $[a_i,b_i]_X$ in $Q_i \cap [x,y]_X$ by setting
$[a_i, b_i]_X = Q_i \cap [b_{i-1},y]_X$ (with the convention $b_0 = x$).
By concatenation of the insulator edges and the $[a_i, b_i]$ we construct a path from $x$ to $y$ in $\cone X$ of length $m' + 2mr_0$ (see Figure~\ref{fig:path}).
This shows $k = m$, and also gives \ref{geodesic2}.
\end{proof}

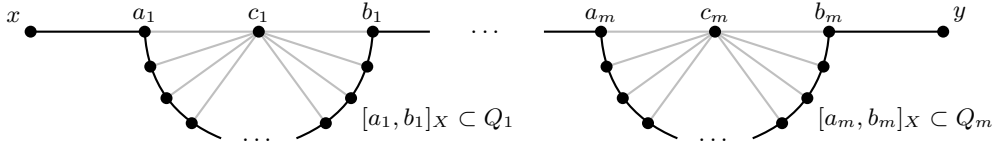
\begin{figure}
\[
\begin{tikzpicture}[scale = 1.5,font=\small,thick]
\coordinate [label=above left:$x$] (x) at (0,0);
\coordinate [label=above:$a_1$] (a1) at (1,0);
\coordinate [label=above:$c_1$] (c1) at (2,0);
\coordinate [label=above:$b_1$] (b1) at (3,0);
\coordinate (b1') at (3.5,0);
\coordinate (ak') at (4.5,0);
\coordinate [label=above:$a_m$] (ak) at (5,0);
\coordinate [label=above:$c_m$] (ck) at (6,0);
\coordinate [label=above:$b_m$] (bk) at (7,0);
\coordinate [label=above right:$y$] (y) at (8,0);
\draw (x)--(a1);
\draw[lightgray] (a1)--(b1) (ak)--(bk);
\draw (b1)--(b1') (ak')--(ak)  (bk)--(y);
\draw[draw=none] (b1) to node[anchor=center]{$\dots$} (ak);
\draw[draw=none,inner sep = 0] (1,0) arc (180:360:1)  
node(p1)[pos=0.1]{} node(p2)[pos=0.2]{} node(p3)[pos=0.3]{} node(p3')[pos=0.4]{}
node(p4')[pos=0.6]{} node(p4)[pos=0.7]{} node(p5)[pos=0.8]{} node(p6)[pos=0.9]{};
\draw (a1) to[bend right=32] (p3'); 
\draw (p4') to[bend right=32,"{$[a_1,b_1]_X \subset Q_1$}",swap] (b1); 
\draw[draw=none] (p3') to node[anchor=center]{$\dots$} (p4');
\foreach \i in {1,...,6}{
\draw[lightgray] (c1)--(p\i);
}
\draw[draw=none,inner sep = 0] (5,0) arc (180:360:1)  
node(q1)[pos=0.1]{} node(q2)[pos=0.2]{} node(q3)[pos=0.3]{} node(q3')[pos=0.4]{}
node(q4')[pos=0.6]{} node(q4)[pos=0.7]{} node(q5)[pos=0.8]{} node(q6)[pos=0.9]{};
\draw (ak) to[bend right=32] (q3'); 
\draw (q4') to[bend right=32,"{$[a_m,b_m]_X \subset Q_m$}",swap] (bk); 
\draw[draw=none] (q3') to node[anchor=center]{$\dots$} (q4');
\foreach \i in {1,...,6}{
\draw[lightgray] (ck)--(q\i);
}
\foreach \v in {x,y,a1,b1,c1,ak,bk,ck,p1,p2,p3,p4,p5,p6,q1,q2,q3,q4,q5,q6} {
\node[bullet] at (\v) {};
}
\end{tikzpicture}
\]
\caption{Associated path. (Recall that the radius $r_0$ satisfies $0 < 2r_0 < 1$, so the figure is combinatorially but not metrically accurate.)} \label{fig:path}
\end{figure}

\begin{proposition}\label{pro:contains_c}
Let $x,y \in X$ be two vertices, and denote $P = [x,y]_X \subset X$.
Let $Q \in \Ql$, with corresponding apex $c \in \cone X$.
Assume that $\diam_X (P \cap Q)> 3\Delta$.
Then every geodesic in $\cone X$ between $x$ and $y$ contains the apex $c$.
\end{proposition}

\begin{proof}
Let $[x,y]$ be a geodesic segment, and let $c_1, \dots, c_m$ be the ordered list of apices contained in $[x,y]$, with associated geodesic $Q_1, \cdots, Q_m$. 
Let $[a_i, b_i]_X \subset Q_i$ be the segments provided by Proposition \ref{pro:geodesic}\ref{geodesic2}. 
Let $i_0$ be the largest integer such that $[a_{i_0},y]_X$ contains $P\cap Q$ (see Figure \ref{figure_proof_QP}). 
Assume that $c$ is not in the list of apices, so that for each $i = 1, \cdots, m$ we have $\diam([a_i,b_i]_X \cap Q) \le \Delta$. The condition $\diam_X (P \cap Q)> 3\Delta$ implies that $a_{i_0+1}$, $a_{i_0+2}$ and $a_{i_0+3}$ belong to $P\cap Q$. 
But then $[a_{i_0+1}, c] \cup [c, a_{i_0+3}]$ is a path of length $2r_0$, in contradiction with the assumption that the path of length $4r_0$ through $a_{i_0+1}, c_{i_0+1}, a_{i_0+2}, c_{i_0+2}, a_{i_0+3}$ was geodesic. 
\end{proof}

 \begin{figure}[h]
	\[
	\begin{tikzpicture}[scale = 0.9,font=\small,thick]
	\draw (-5,0) node[bullet]{} node[below] {$x$};
	\draw (4.5,0) node[bullet]{} node[below] {$y$};
	\draw (1.5,2) node[right] {$Q$};
	\draw (3.5,0) node[below] {$P$};
	\draw[lightgray] (-3,-1.5) -- (-4,0);
	\draw[lightgray] (-3,-1.5) -- (-3.5,0);
	\draw[lightgray] (-3,-1.5) -- (-3,0);
	\draw[lightgray] (-3,-1.5) -- (-2.5,0);
	\draw[lightgray] (-3,-1.5) -- (-2,0);
	\draw[lightgray] (-1,-1.5) -- (-2,0);
	\draw[lightgray] (-1,-1.5) -- (-1.5,0);
	\draw[lightgray] (-1,-1.5) -- (-1,0);
	\draw[lightgray] (-1,-1.5) -- (-0.5,0);
	\draw[lightgray] (-1,-1.5) -- (0,0);          
	\draw[lightgray] (-1,-1.5) -- (0,0);
	\draw[lightgray] (1,-1.5) -- (0,0);          
	\draw[lightgray] (1,-1.5) -- (0.5,0);   
	\draw[lightgray] (1,-1.5) -- (1,0);   
	\draw[lightgray] (1,-1.5) -- (1.5,0);   
	\draw[lightgray] (1,-1.5) -- (2,0);   
	\draw[lightgray] (-.5,1.5) -- (2,0); 
	\draw[lightgray] (-.5,1.5) -- (1.5,0);   
	\draw[lightgray] (-.5,1.5) -- (1,0);  
	\draw[lightgray] (-.5,1.5) -- (0.5,0);  
	\draw[lightgray] (-.5,1.5) -- (0,0);  
	\draw[lightgray] (-.5,1.5) -- (-0.5,0);  
	\draw[lightgray] (-.5,1.5) -- (-1,0);  
	\draw[lightgray] (-.5,1.5) -- (-1.5,0);  
	\draw[lightgray] (-.5,1.5) -- (-2,0);  
	\draw[lightgray] (-.5,1.5) -- (-2.5,0);  
	\draw[lightgray] (-.5,1.5) -- (-3,0);  
	\draw[lightgray] (-.5,1.5) -- (-3,0.5);  
	\draw[lightgray] (-.5,1.5) -- (-3,1); 
	\draw[lightgray] (-.5,1.5) -- (2.5,0.5);  
	\draw[lightgray] (-.5,1.5) -- (2.5,0);  
	\draw[lightgray] (-.5,1.5) -- (2.5,1); 
	\draw (-4,0) node[below] {$a_{i_0}$} node[bullet]{};
	\draw (-2,0) node[below] {$a_{i_0+1}$} node[bullet]{};
	\draw (0,0) node[below] {$a_{i_0+2}$}  node[bullet]{};
	\draw (2,0) node[below] {$a_{i_0+3}$}  node[bullet]{};
	\draw (-.5,1.5) node[bullet]{} node[above] {$c$};
	\draw (-3,-1.5) node[bullet]{} node[below] {$c_{i_0}$};
	\draw (-1,-1.5) node[bullet]{} node[below] {$c_{i_0+1}$};
	\draw (1,-1.5) node[bullet]{} node[below] {$c_{i_0+2}$};
	\draw[thick] (-5,0)-- (4.5,0);
	\draw[<-, thick] (2.5,2)-- (2.5,0);
	\draw[thick] (-3,2)-- (-3,0);
	\end{tikzpicture}\]
	\caption{A geodesic between $x$ and $y$ has to pass through $c$ when $\diam_X(Q\cap P)> 3\Delta$. \label{figure_proof_QP}}
\end{figure}
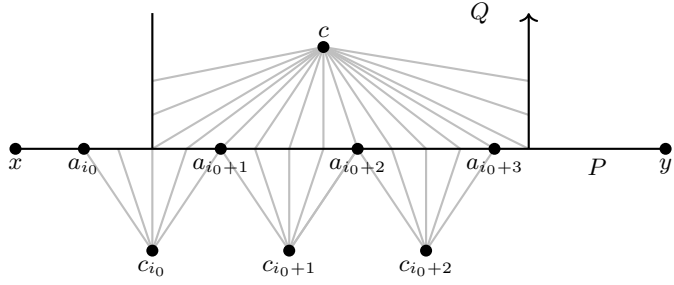

\subsection{Windmills}

We assume Set-up \ref{setup:family}.

\begin{definition}Let $A\subset \mathring{X}$ be a subgraph.
\begin{itemize}
\item The group $G_A$ is the group generated by all $G_c$ where $c$ runs over all apices in $A$:
\[G_A=\langle G_c\mid c\in A \cap \Cl\rangle.\]
\item The \textit{saturate} of $A$, denoted by $\Sat(A)$, is the minimal set containing $A$ that is invariant under the action of $G_A$.
The set $A$ is \textit{saturated} if $A=\Sat(A)$.
\item $A$ is \emph{complete} if for each apex $c$ associated with a geodesic $Q \in \Ql$, either one of the conditions $c \in A$ or $Q\subset A$ implies $\cone Q \subset A$.
\item $A$ is \emph{quasiconvex} if $A$ is connected and $A \cap X$ is connected.
\end{itemize}
\end{definition}

\begin{lemma} \label{lem:xqc}
Let $A \subset \cone X$ be a quasiconvex subgraph, $Q \in \Ql$ an adjacent geodesic to $A$, and $c$ the apex of $\cone Q$.
Let $y \in A \cap X$ be a vertex, and $p$ be the projection of $y$ on $Q$ in the tree $X$.
Then $p \in A$, and there exists a geodesic $[y,c]$ of the form $[y,p] \cup [p,c]$.
Moreover, for any $p' \in Q$ such that there exists a geodesic $[y,c]$ of the form $[y,p'] \cup [p',c]$, we have $d_X(p,p') \le \Delta$.
\end{lemma}

\begin{proof}
Any geodesic from $y$ to $c$ is of the form $[y,q] \cup [q,c]$ with $q \in Q$.
Since $[y,p]_X$ is a subsegment of $[y,q]_X$, any covering family of the non insulator edges in $[y,q]_X$ also is a covering family for $[y,p]_X$, and by Proposition \ref{pro:geodesic} we get $\dist(y,p) \le \dist(y,q)$.
So $\dist(y,p) = \dist(y,q)$, and $[y,p] \cup [p,c]$ is geodesic.
Let $q'$ be any vertex in $Q \cap A$.
Again $[y,p]_X$ is a subsegment of $[y,q']_X$, and by quasiconvexity of $A$ we get $p \in A$. Moreover, let $p' \in Q$ such that there exists a geodesic $[y,c]$ of the form $[y,p'] \cup [p',c]$. Consequently, $\dist(y,p')=\dist(y,p)$ and Proposition \ref{pro:geodesic} implies that there exists a geodesic $Q'$ different from $Q$ containing both $p$ and $p'$. Consequently $\dist(p,p')\leq \Delta$.
\end{proof}

\begin{remark}\label{rmq:quasiconvexity}
Lemma \ref{lem:xqc} justifies why we call quasiconvex a subgraph connected such that its intersection with $X$ is connected. It is indeed $\Delta$-quasiconvex in the sense of \cite[Definition 3.3]{DGO}. 
If moreover we add the condition that the subgraph is complete then it is $r_0$-quasiconvex.

\end{remark}

\begin{definition}[{Windmill, \cite[Definition 5.11]{DGO}}]
Assume Set-up \ref{setup:family}. A \emph{windmill} is a subgraph $W \subset \cone X$ satisfying
\begin{enumerate}[(W1)]
\item\label{WM_1}
$W$ is quasiconvex, saturated and complete;
\item\label{WM_adjacent}
For any $Q \in \Ql$ adjacent to $W$, we have $\diam_X(Q \cap W)\leq 2\Delta$;
\item\label{WM_free_product}
$G_W$ is a free product of some groups among the $G_c$, $c \in C_W$;
\item\label{WM_length}
Every non-trivial element $f$ in $G_W$ is loxodromic for the action on $X$, with translation length $\lX(f) \ge \lX(g^n)$, and equality if and only if $f$ is conjugate to $g^n$.
\end{enumerate}
\end{definition}

\begin{example} \label{exple:1st_windmill}
Let $c \in \mathcal{C}$ be an apex with associated geodesic $Q$.
Then $W = \cone Q$ is a windmill, with $G_W = G_c$.
\end{example}

The main result we want to prove about windmills is the following:

\begin{proposition}[Growing windmills, {\cite[Proposition 5.12]{DGO}}]
\label{pro:growing}
Assume Set-Up \textup{\ref{setup:family}}.
For any windmill $W \subsetneq \cone X$, there exists a windmill $W'$ strictly containing $W$.
\end{proposition}

\begin{remark}
	In fact the proof will show that
	\[
	G_{W'} = G_W \ast (\ast_{c \in S} G_c)
	\]
	for some set $S \subset \Adj(W)$.
\end{remark}

We will need the following two technical facts about windmills.

\begin{proposition} \label{pro:by c}
Assume Set-up \textup{\ref{setup:family}}.
Let $W$ be a windmill, and $c$ the apex of an adjacent geodesic $Q$.
Let $x_1, x_2\in X$ such that there exist geodesic segments between $x_i$ and $c$ meeting $W \cap X$, $i = 1,2$.
Then:
\begin{enumerate}
\item For any $f \in G_c\setminus\{\id\}$, any geodesic segment between $x_1$ and $f(x_2)$ contains~$c$. 
\item In particular, for any choice of geodesic segments $[x_1,c]$ and $[c,f(x_2)]$, the concatenation $[x_1,c]\cup[c,f(x_2)]$ also is geodesic.
\end{enumerate}
\end{proposition}

\begin{proof}
\begin{enumerate}
\item
For $i = 1,2$, by assumption there exists a geodesic segment $[x_i,y_i] \cup [y_i, c]$ with $y_i \in W \cap X$.
By Lemma \ref{lem:xqc}, there exists a geodesic segment of the form $[y_i, p_i] \cup [p_i, c]$, where $p_i \in Q \cap W$ is the projection of $y_i$ on $Q$ in the tree $X$.
By Property \ref{WM_adjacent}, we have $d_X(p_1, p_2) \le 2\Delta$.
Then
\[
d_X(p_1, f(p_2)) \ge d_X(p_2, f(p_2)) -2\Delta = \lX(g^n) - 2\Delta > 7\Delta - 2\Delta = 5\Delta.
\]
Denote by $\bar{x_i}$ the projection of $x_i$ on $Q$ in the tree $X$. 
Because $[x_i,\bar{x_i}]_X\subset [x_i,p_i]_X$ and $[x_i,p_i]\cup[p_i,c]$ is geodesic, by Proposition \ref{pro:geodesic} there is also a geodesic of the form $[x_i,\bar{x_i}]\cup[\bar{x_i},c]$, and either $p_i=\bar{x_i}$ or they belong to a geodesic $Q'\in \Ql\setminus \{Q\}$. Consequently, $\dist_X(\bar{x_i},p_i)\leq \Delta$, and 
\[
\diam_X(Q\cap [x_1,f(x_2)]_X) = d_X(\bar x_1, f(\bar x_2)) \geq d_X(p_1,f(p_2))-2\Delta  \geq 3\Delta.
\] 
We conclude by Proposition \ref{pro:contains_c}.
\item
For any choice of geodesic segments $[x_1,c]$ and $[c,f(x_2)]$, the sum of their lengths is the same than the geodesic $[x_1, f(x_2)]$ passing through $c$ found in the previous point, so $[x_1,c]\cup [c,f(x_2)]$ also is geodesic.
\qedhere
\end{enumerate}
\end{proof}

\begin{lemma} \label{lem:adj c1 c2}
Let $W$ be a windmill, and $c_1, c_2 \in \Adj(W)$ with associated geodesics $Q_1, Q_2$.
Then for $i =1,2$ there exist $q_i \in Q_i \cap W$ such that 
$[c_1, q_1] \cup [q_1, q_2] \cup [q_2, c_2]$ is geodesic. 
\end{lemma}

\begin{proof}
If $Q_1 \cap Q_2 \neq \emptyset$, then by \ref{WM_1} we can take $q_1 = q_2 \in Q_1 \cap Q_2 \cap W$ (in a tree any collection of pairwise intersecting subtrees admits a common point).
If $Q_1 \cap Q_2 = \emptyset$ we take $q_1, q_2$ the endpoints of the unique segment from $Q_1$ to $Q_2$ in $X$. 
By connectedness of $W \cap X$, the assumptions $Q_i \cap W \neq \emptyset$ imply $[q_1, q_2]_X \subset W$.
Let $p_i\in Q_i$. By definition of the $q_i$'s, $[q_1,q_2]_X\subset [p_1,p_2]_X$ so any covering of $[p_1,p_2]_X$ by geodesics $Q\in \Ql$ is also a covering of $[q_1, q_2]_X$. 
Hence by Proposition \ref{pro:geodesic}\ref{geodesic1}, $\dist(q_1,q_2)\leq \dist(p_1,p_2)$, consequently $[c_1,q_1]\cup[q_1,q_2]\cup [q_2,c_2]$ is geodesic. 
\end{proof}

\subsection{The tree $T_W$ of a windmill}
\label{sec:TW}
In this section we gather some preliminary material before proving Proposition \ref{pro:growing} and Theorem \ref{thm:SCT}.
Assume Set-Up \ref{setup:family}, let $W$ be a windmill, and assume $\Adj(W) \neq \emptyset$.
We set
\[
W'= \Sat \left( W\cup \bigcup_{c \in \Adj(W)} \cone Q_c \right).
\]
Note that $W'$ will correspond to the larger windmill of Proposition \ref{pro:growing} in the case where $\Adj(W)\neq \emptyset$.

The group $G_{W'}$ is the group generated by $G_W$ and all $G_c$, where $c \in \Adj(W)$. 

We define a bicolored graph $T_W$ as follow.
The two set of vertices are indexed respectively by $\{gc \mid g \in G_{W'}, c \in \Adj(W)\}$ and by $\{gW \mid g \in G_{W'}\}$.
We put an edge between two vertices $v_1, v_2$ if there exist $g \in G_{W'}$ and $c \in \Adj(W)$ such that $v_1 = gc$ and $v_2 = gW$.
There is a natural action of $G_{W'}$ on $T_W$ by left translation.

\begin{setup} \label{setup:path}
A \emph{path} from $W$ in the graph $T_W$ is a sequence of vertices of the form:
\[
W, c_1, g_1 W, c_2, g_2 g_1 W, \dots, c_m, g_m \dots g_1 W,
\]
where for each $i =1, \dots, m$, $c_i$ is adjacent to $g_{i-1} \dots g_1 W$ (with the convention $g_0 = 1$), $g_i \in G_{c_i} \setminus \{1\}$, and $c_i \neq c_{i-1}$ for all $2 \le i \le m$.
The conditions $g_i \neq 1$ and $c_i \neq c_{i-1}$ insure that the path is locally injective.
Observe that by construction each $c_i$ is also adjacent to $g_{i} \dots g_1 W$.
\end{setup}

\begin{lemma} \label{lem:distWhW}
Given a path as in Set-Up \textup{\ref{setup:path}}, for each $1 \le k \le m$ we have
\[\dist_X(W,g_k g_{k-1}\dots g_2 g_1W) > k \lX(g^n) - (3k-1)\Delta.\]
\end{lemma}

\begin{proof}
We denote by $x$ and $x_1$ the endpoints of the geodesic segment in the tree $X$ joining $W\cap X$ and $g_1W\cap X$.
Because $Q_{c_1}$ is adjacent to both $W$ and $g_1W$, $x\in W\cap Q_{c_1}$ and $x_1\in Q_{c_1}\cap g_1W $. Moreover, $\dist_X(x_1,g_1x)\leq \diam(Q_{c_1}\cap g_1W)\leq 2\Delta$. Hence, we get the expected formula for $k = 1$: 
\[\dist_X(W,g_1W)\geq\lX(g^n)-2\Delta.\] 
By the same argument, for any $k$, $\dist_X(g_{k-1}\dots g_2g_1W,g_kg_{k-1}\dots g_2g_1W)\geq\lX(g^n)-2\Delta$, where this distance is realized by a subsegment of $Q_{c_k}$.
Since for each $k \ge 1$ we have $\diam(Q_{c_k}\cap Q_{c_{k+1}})<\Delta$ we obtain
\begin{align*}
\dist_X(W,g_k g_{k-1}\dots g_2 g_1W) &> k(\lX(g^n)-2\Delta) - (k-1) \Delta \\
&= k \lX(g^n) - (3k-1)\Delta.  \qedhere
\end{align*}
\end{proof}

\begin{lemma} \label{lem:through_ci}
Given a path as in Set-Up \textup{\ref{setup:path}}, let $x_0 \in W \cap X$ and $x_m \in g_m \dots g_1 W \cap X$ be two vertices.
Then any geodesic segment $[x_0, x_m]$ passes through $c_m$.
\end{lemma}

\begin{proof}
We are going to prove the following assertion by induction on $k$:
For any $x_k \in g_k \dots g_1 W \cap X$, any geodesic segment $[x_0, x_k]$ passes through $c_k$.

For $k = 1$, this follows directly from Proposition \ref{pro:by c}.
Now assume $k \ge 2$, and that the property is true for all indices between $1$ and $k-1$.
By Lemma \ref{lem:adj c1 c2} there exists $x_{k-1} \in g_{k-1} \dots g_1 W \cap Q_{c_k}$ such that $[c_{k-1}, x_{k-1}] \cup [x_{k-1}, c_k]$ is geodesic.
Let $[x_0, x_{k-1}]$ be a geodesic segment.

If we can prove that the concatenation  $[x_0, x_{k-1}] \cup [x_{k-1},c_k]$ also is geodesic, then we conclude by Proposition \ref{pro:by c} that any geodesic from $x_0$ to $x_k$ contains $c_k$ and we are done.

So by contradiction, assume that $[x_0, x_{k-1}] \cup [x_{k-1},c_k]$ is not geodesic.
Then, there exists a geodesic of the form $[x_0, y] \cup [y, c_k]$ with $y \in Q_{k}$ and $\dist(x_0,y) < \dist(x_0,x_{k-1})$.
This implies $\dist(x_0,y) = \dist(x_0,x_{k-1}) - 2r_0$, so that $[x_0, y] \cup [y,c_k] \cup [c_k,x_{k-1}]$ is geodesic.
But then by induction hypothesis $[x_0,y]$ must pass through $c_{k-1}$, so we also have 
\[
\dist(c_{k-1},y) = \dist(c_{k-1},x_{k-1}) - 2r_0.
\]
But this contradicts the fact that the following two concatenations of segments are geodesics from $c_{k-1}$ to $c_k$:
\[
[c_{k-1}, x_{k-1}] \cup [x_{k-1}, c_k]
\text{ and }
[c_{k-1}, y] \cup [y, c_k]
\qedhere
\]
\end{proof}

\begin{corollary}\label{cor:tree_windmill}
The graph $T_W$ is a tree.
\end{corollary}

\begin{proof}
By contradiction, assume that $\gamma$ is an embedded circle in the graph $T_W$.
By transitivity of the action of $G$, we can assume that one of the vertex of $\gamma$ is $\id W$.
Then we get a path in $T_W$ as in Set-Up \ref{setup:path}:
\[
W, c_1, g_1 W, c_2, g_2 g_1 W, c_3, \dots, c_m, g_m \dots g_1 W,
\]
with $g_m \dots g_1 W = W$, and $m \ge 2$.
By applying Lemma \ref{lem:through_ci} to $x_0 = x_m \in W \cap X$, we get a contradiction: the constant path $[x_0,x_0]$ does not pass through $c_m$.
\end{proof}

\begin{corollary} \label{cor:adjQ1Q2}
Let $c, c'$ be the apices of two geodesics $Q_c, Q_{c'} \in \Ql$ contained in $W'$. 
If $Q_c \cap Q_{c'} \neq \emptyset$, then there exists $g \in G_{W'}$ such that $g(Q_c)$, $g(Q_{c'})$ are both adjacent to $W$. 
\end{corollary}

\begin{proof}
If $c = c'$, the conclusion is direct.
Otherwise we consider the geodesic path from $c$ to $c'$ in the tree $T_W$.
Up to the action of $G_{W'}$, we can assume that the path is of the form $c, W, c_1, \dots,c_m, g_m\dots g_1W, c'$. 
We want to prove that the path as length 2, that is, $m=0$ and $g_m\dots g_1W = W$.

By contradiction assume that $g_m\dots g_1W\neq W$, with $m \ge 1$.
We can apply Lemma \ref{lem:through_ci} to $x_0 \in Q_c \cap W$ and $x_m \in g_m \cdots g_1 W \cap Q_{c'}$, and we get that any geodesic $[x_0,x_m]$ should pass through $c_m$.
Observe that $x_0 \not\in Q_{c'}$, $x_m \not\in Q_c$ and $Q_c \cap Q_{c'} \cap [x_0,x_m]_X \neq \emptyset$.
Choose $y \in Q_c \cap Q_{c'} \cap [x_0,x_m]_X$.
We have a path of length $4r_0$ from $x_0$ to $x_m$, passing through $c, y, c'$, which can not be geodesic because it does not pass through $c_m$. 
So $\dist(x_0,x_m) =2r_0$.
Consequently, there exists a geodesic $Q\in \Ql$ distinct from both $Q_{c}$ and $Q_{c'}$, and containing $x_0$ and $x_m$. Hence, we get:
\[d_X(x_0, x_1) \le d_X(x_0, y) + d_X(y, x_1) \le \diam_X(Q \cap Q_c) + \diam_X(Q \cap Q_{c'}) \le 2\Delta,\]
in contradiction with Lemma \ref{lem:distWhW}. 
\end{proof}

\begin{remark}\label{remark_axe-dans-W}
By \ref{WM_length}, any $f\in G_W\setminus\{\id\}$ is loxodromic for the action on $X$. 
Moreover, we claim that $\axe(f)\subset W$.
Indeed, $f$ preserves $W$.
Let $x\in W\cap X$ and $\bar{x}$ its projection in the tree on $\axe(f)$.
By quasiconvexity of $W$, $[x,f^n(x)]_X$ is in $W$ for any $n \in \Z$, so the subsegment $[\bar x, f^n(\bar x)] \subset \axe(f)$ is also in $W$, proving the claim.
\end{remark}

\begin{lemma}\label{lemma_W4}
Every non-trivial element $f$ in $G = G_{W'}$ is loxodromic for the action on $X$, with translation length $\lX(f) \ge \lX(g^n)$, and equality if and only if $f$ is conjugate to $g^n$.
\end{lemma}

\begin{proof}
We use the action of $G$ on the bicolored tree $T_W$ (see Corollary \ref{cor:tree_windmill}).
Let $f$ be a non trivial element in $G$.

If the action of $f$ on $T_W$ is elliptic, we have two cases.
If $f$ is conjugate to an element in $G_W$, then we get the conclusion by property \ref{WM_length}.
Otherwise $f$ lies is some $G_c$ and acts on $X$ as a loxodromic element with axis $Q_c$.
In this case, $f$ is conjugated to some power of $g^n$, so we get $\lX(f) \ge \lX(g^n)$, with equality if and only if $f$ is conjugate to $g^n$.
    
If the action of $f$ on $T_W$ is loxodromic, we consider the axis of $f$ in $T_W$.
Up to conjugacy, we can assume that $W$ is a vertex of this axis.
We consider the path of length $2m$ between $W$ and $fW$ as in Set-up \ref{setup:path}: 
\[W, c_1, g_1 W, c_2, g_2 g_1 W, c_3, \dots, c_m, g_m \dots g_1 W=fW,\] where $2m$ is the translation length of $f$ on $T_W$. 
This path is included in $\axe(f)$.

If $m=1$ then there exists $w \neq 1\in G_W$ such that $f=g_1w$. 
By Remark \ref{remark_axe-dans-W}, $\axe(w)\subset W$ so by \ref{WM_adjacent} the diameter of the intersection $\axe(w)\cap \axe(g_1)$ is less than $2\Delta$.
Consequently, using \ref{WM_length} and that $\lX(g^n)>7\Delta$ :
\[\lX(g_1w)\geq\lX(g_1)+\lX(w)-2\diam(\axe(g_1)\cap \axe(w))\geq \lX(g^n)+(\lX(g^n)-4\Delta)>\lX(g^n).\]

Now assume that $m\geq 2$.
Consider $[y,y']_X$ realizing the distance in the tree between $W\cap Q_{c_1}$ and $f(W)\cap f(Q_{c_{1}})$. Choose $x\in  W\cap Q_{c_1}$ such that $f(x)=y'$. Recall that a point $z$ belongs to the axis of $f$ in the tree $X$ if and only if $f(z)\in [z,f^2(z)]$. Hence, if $f(x)\notin Q_{c_m}$ then $x$ is on the axis of $f$ for the action on $X$ and $\l_X(f)\geq \dist_X(W,g_m g_{m-1}\dots g_2 g_1W)$. If $f(x)\in Q_{c_m}$ then \[\diam_X\left( [x,f(x)]_X\cap [f(x),f^2(x)]_X\right)\leq \diam_X(Q_{c_m}\cap f(Q_{c_{1}}))\leq \Delta,\]  and so $ \lX(f)\geq \dist_X(W,g_m g_{m-1}\dots g_2 g_1W)-2\Delta$. Consequently, in the two cases we have that, using Lemma \ref{lem:distWhW}, the translation length of $f$ is:
	\begin{multline*}
	\lX(f)\geq \dist_X(W,g_m g_{m-1}\dots g_2 g_1W)-2\Delta> m \lX(g^m) - (3m-1)\Delta -2\Delta\\
	=\lX(g^n)+(m-1)(\lX(g^n)-3\Delta)-4\Delta>\lX(g^n)
	\end{multline*}
	when $\lX(g^n)>7\Delta$.  
\end{proof}

\subsection{Proofs}

\begin{proof}[Proof of Proposition \ref{pro:growing}]
First we observe that if $\Adj(W)$ is empty, then the proof of Proposition \ref{pro:growing} is direct.
The completion assumption implies that $W \cap X \subsetneq X$, so that we get a strictly larger set by defining
\[W'=W\cup \{e\subset X\mid e \text { insulator edge with one end in }W\}.\]
By construction, $W'$ is quasiconvex and complete. $W'$ does not contain any new apex so conditions \ref{WM_1}, \ref{WM_free_product} and \ref{WM_length} are satisfied. 
Moreover $W$ does not have any adjacent geodesic, so for any $Q\in \Ql$ adjacent to $W'$, the intersection $Q\cap W'$ is either empty or reduced to a singleton, which gives \ref{WM_adjacent}.

Now we assume that $\Adj(W)$ is not empty, and we define $W'$ as in \S\ref{sec:TW}.
By definition $W'$ is quasi-convex, saturated and complete so it satisfies property \ref{WM_1} of the definition of a windmill.
Property \ref{WM_free_product} is obtained using the Bass-Serre theory (Proposition \ref{pro:bass serre}) because by Corollary \ref{cor:tree_windmill}, $G_{W'}$ acts on the tree $T_W$ with trivial stabilizers of edges, and stabilizers of vertices are either conjugates of $G_W$ or conjugates of $G_c$ for $c\in\Adj(W)$. Hence, using the fact that $W$ is a windmill, the group $G_{W'}$ is a free product of some groups among the $G_c$, $c\in C_{W'}$.
Lemma~\ref{lemma_W4} gives the axiom \ref{WM_length} so it remains to prove that $W'$ satisfies \ref{WM_adjacent}.

By contradiction, assume that there exists $Q \in \Ql$ an adjacent geodesic to $W'$ with $\diam_X(Q \cap W') > 2 \Delta$.
This means that there exists three geodesics $Q_1, Q_2, Q_2'$ in $W'$, which are orbits of adjacent geodesics to $W$, such that there exists a geodesic subsegment $[v_2', v_1'] \cup [v_1', v_1] \cup [v_1,v_2] \subset Q$ of length $> 2\Delta$ not meeting the orbit of $W$, such that $v_1, v_1' \in Q_1$, $v_1', v_2' \in Q_2'$, $v_1, v_2 \in Q_2$.
By Corollary \ref{cor:adjQ1Q2}, using the action of $G_{W'}$ we can assume that $Q_1$, $Q_2$ are adjacent to $W$.
Then there exists $q \in W \cap Q_1 \cap Q_2$ such that $q,v_2, v_1'$ form a tripod in $X$ with branch point $v_1$ (see Figure \ref{fig:proofW2}).

We have $\diam_X([q,v_1]) \le \Delta$ and $\diam_X([v_1,v_1']) \le \Delta$, because $[q,v_1] \subset Q_1 \cap Q_2$ and $[v_1,v_1'] \subset Q_1 \cap Q$.
But then by a similar argument $v_1'$ should be $\Delta$-close to a point $q' \in g' W$, for some $g'\in G_{W'}$ with axis $Q_1$, contradicts the fact that $\dist(W, g'W) > 3 \Delta$.
\end{proof}

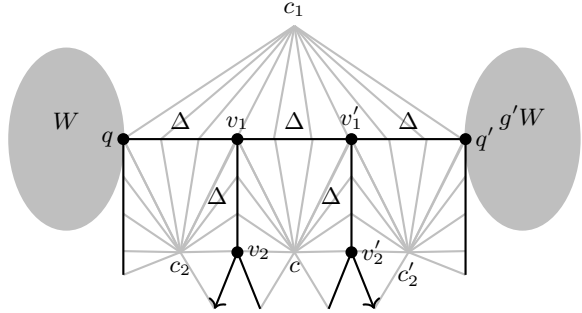
\begin{figure}[h]
\[
\begin{tikzpicture}[scale = 1.5,font=\small,thick]
\draw[lightgray, fill=lightgray] (-0.5,0) ellipse (0.5 and 0.8);
\coordinate [label=above:$W$] (W) at (-0.5,0);
\draw[lightgray, fill=lightgray] (3.5,0) ellipse (0.5 and 0.8);
\coordinate [label=below:$c_2$] (c2) at (0.5,-1); 
\coordinate [label=below:$c_2'$] (c2') at (2.5,-1); 
\foreach \v in {0,-0.33, ..., -1} {
\draw[lightgray] (c2) -- (0,\v);
\draw[lightgray] (c2) -- (1,\v);
\draw[lightgray] (1.5,-1) -- (1,\v);
\draw[lightgray] (1.5,-1) -- (2,\v);
\draw[lightgray] (c2') -- (2,\v);
\draw[lightgray] (c2') -- (3,\v);
}
\foreach \v in {0,0.33, ..., 1} {
\draw[lightgray] (c2) -- (\v,0);
}
\foreach \v in {1,1.33, ..., 2} {
\draw[lightgray] (1.5,-1) -- (\v,0);
}
\foreach \v in {2,2.33, ..., 3} {
\draw[lightgray] (c2') -- (\v,0);
}
\coordinate [label=above:$c_1$] (c1) at (1.5,1); 
\foreach \v in {0,0.33, ..., 3} {
\draw[lightgray] (c1) -- (\v,0);
}
\coordinate (Qstart) at (1.2,-1.5);
\coordinate (Q2start) at (0,-1.2);
\coordinate (Q2'start) at (3,-1.2);
\coordinate  (Qend) at (1.8,-1.5);
\coordinate (Q2end) at (0.8,-1.5);
\coordinate  (Q2'end) at (2.2,-1.5);
\draw[lightgray] (Q2end) -- (c2) -- (Q2start);
\draw[lightgray] (Qend) -- (1.5,-1) -- (Qstart);
\draw[lightgray] (Q2'end) -- (c2') -- (Q2'start);

\coordinate [label=above:$g'W$] (g'W) at (3.5,0);
\coordinate [label=left:$q$] (q) at (0,0);
\coordinate [label=above:$v_1$] (v1) at (1,0);
\coordinate [label=above:$v_1'$] (v1') at (2,0);
\coordinate [label=right:$q'$] (q') at (3,0);
\coordinate [label=right:$v_2$] (v2) at (1,-1);
\coordinate [label=right:$v_2'$] (v2') at (2,-1);
\draw (Q2start)--(q) to["$\Delta$"] (v1) to["$\Delta$"] (v1') to["$\Delta$"] (q')--(Q2'start) 
      (v2) to["$\Delta$"] (v1) 
      (v2') to["$\Delta$"] (v1')
      (Qstart)--(v2)
      (Qend)--(v2');
\draw[->] (v2)--(Q2end);
\draw[->] (v2')--(Q2'end);
\coordinate [label=below:$c$] (c) at (1.5,-1); 
\foreach \v in {q,q',v1,v1',v2,v2'} {
\node[bullet] at (\v) {};
}
\end{tikzpicture}
\]
\caption{An impossible configuration of the geodesics $Q,Q_1,Q_2,Q_2$, with respective apices $c,c_1,c_2,c_2'$.} \label{fig:proofW2}
\end{figure}

\begin{proof}[Proof of Theorem \ref{thm:SCT}]
We use the notation of Set-Up \ref{setup:family}.
We consider the collection of all windmills, which is not empty by Example \ref{exple:1st_windmill}, and we consider the partial order given by inclusion.
By Zorn's Lemma, there exists a maximal windmill $W$.
By Proposition \ref{pro:growing}, we must have $W = \cone X$.
Then Properties \ref{WM_free_product} and \ref{WM_length} give the assertions of the theorem.  
\end{proof}

\subsection{Further comments} \label{sec:comment}

In this section, we highlight some modifications we made compared to \cite{DGO}.

A first difference is the choice of the radius $r_0$ in the cone-off construction.
In \cite[\S 5.3]{DGO}, starting from a Gromov-hyperbolic metric space $X$, in order for the cone-off $\cone X$ to also be hyperbolic the radius $r_0$ has to be larger that a universal constant $r_U >5\times 10^{12}$.
In contrast, we chose to work with a small radius $0 < r_0 <\frac12$, since it allows us an easy description of geodesics in $\cone X$, taking advantage of the fact that $X$ is a simplicial tree (Proposition \ref{pro:geodesic}).
 
In \cite[Definition 5.1]{DGO}, the family of subgroups $\{G_Q\mid Q\in \Ql\}$ from Set-up \ref{setup:family} is called a \emph{rotating family}.
They also define a notion of very rotating family, which is a local condition about the action of $G_c$ on points close to the apex $c$, and they observe that it implies a global very rotating condition \cite[Lemma 5.5]{DGO}. 
Transposing in our context, we can make the following definition:
 
\begin{definition} 
Let $\{G_Q\mid Q\in \Ql\}$ be the family of groups from Set-up \ref{setup:family}.
\begin{itemize}
\item The family is \textit{locally very rotating} if for every $Q \in \Ql$, for every $x\in Q$ and for every $h\in G_{Q}\setminus\{\id\}$, any geodesic $[x,hx]$ contains $c_Q$.
\item We say that the family is \textit{globally very rotating } if the same property holds for  any $x \in X$.
\end{itemize}
\end{definition}

By our choice of working with a small $r_0$, $\{G_Q\mid Q\in \Ql\}$ is automatically a locally very rotating family. 
Morever using Proposition \ref{pro:contains_c} one can easily show that the family is globally very rotating as soon as $\lX(g^n) > 3\Delta$.
So even if we did not use this terminology of rotating family, here we followed quite closely the line of argument of \cite{DGO}. 
  
Finally we say a word about our definition of \emph{windmill}, where we had to adapt the part of the definition that relies on the hyperbolicity constant $\delta$.
We chose to work instead with the constant $\Delta$, which bounds the diameter of intersections of axes in our family of conjugate loxodromic isometries.
First we used in \ref{WM_1} an \emph{ad hoc} definition of quasiconvexity, which seems natural in our context and bears some resemblance with the general notion of quasiconvexity, as noted in Remark \ref{rmq:quasiconvexity}.
Second we put in \ref{WM_adjacent} a bound on the intersection of an adjacent geodesic with the windmill, which in \cite{DGO} was a consequence of the Gromov-hyperbolicity of $\cone X$ (\cite[Proof of Lemma 5.15]{DGO}).  
Finally observe that in our axiom \ref{WM_length} we ask for large translation lengths with respect to the action on the initial tree $X$, whereas in \cite{DGO} they ask for the similar condition on the cone-off $\cone X$. 
The bridge between the two is essentially our Lemma \ref{lemma_W4}.
 
\section{Group acting on a tree, following \cite{CL}}
\label{sec:CL}

In this section we prove assertion \ref{SCT2} of Theorem \ref{thm:SCT}, following \cite{CL}, from which assertion~\ref{SCT3} also directly follows.
This strategy of proof does not seem to easily provide assertion \ref{SCT1}, so we do not attempt to prove it in this paper.
Observe that once assertion \ref{SCT2} is established it follows that the normal subgroup $\llb g^n \rrb$ is a free group, because any group that acts freely on a tree is a free group  \cite[Theorem 3.3.4]{Serre}.
However the fact that one can choose a collection of conjugates of $g^n$ as a free basis is the difficult part in assertion \ref{SCT1}. 

\subsection{Admissible presentations}

Let $G \action X$ be a group acting on a simplicial tree, and $g \in G$ a loxodromic WPD element.
By Proposition~\ref{pro:WPD=>SC}, up to passing to some power we can assume that $g$ is tight.
In the following definitions we work with a fixed iterate $g^n$ of $g$, and we note $\ell = n\ell(g)$ the translation length of $g^n$.

\begin{definition}[Relator and neutral segment, {\cite[\S 2.4.1]{CL}}]\hspace{0.3cm}
\begin{itemize}
\item 
Let $r \in \Q_{>0}$. 
A \emph{relator of size $\geq r$} is an oriented segment $[x,y]\subset X$ with $\dist(x,y) \geq r$ and such that there exists a conjugate $f$ of $g^n$ with $[x,y] \subset \axe(f)$.
\item Given such a relator, up to replacing $f$ by $f^{-1}$, we can assume that $x \not\in [y, f(x)]$.
In that case, we say that $f$ is the \emph{support} of the relator $[x,y]$ (and so $f^{-1}$ is the support of the relator $[y,x]$).
\item We say that a segment $[x,y]$ is \emph{neutral} if for any relator of size $\geq r$ contained in $[x,y]$ we have $r \le \frac{1}{2}\l$.
Observe that this notion is stable under the action of $G$: if $[x,y]$ is neutral and $h \in G$, then $[hx,hy]$ is also neutral.
\end{itemize}
\end{definition}

\begin{remark}
In \cite{CL} a relator is called a \emph{piece}.
This is the result of an unfortunate last minute change of terminology (``relator'' was our working vocabulary, and was still used in the first version on arXiv...), as a piece according to classical small cancellation theory would rather be a segment contained in the intersection of \emph{two} distinct axes of conjugates of $g^n$.
The ``small'' of \emph{small} cancellation theory refers to the fact that this intersection is small in comparison with the translation length of $g^n$. This small ratio corresponds to the constant $\eps$ of Definition \ref{def:SC} (applied to $g^n$), and also to the constant $\frac{1}{12}$ that will appear in section \ref{sec:Greendlinger}.
\end{remark}

\begin{definition}[Admissible presentation, {\cite[\S 2.4.2]{CL}}] \label{def:admissible}
By definition any $h \in \llb g^n \rrb$ admits a \emph{presentation} as a product $h =h_mh_{m-1} \dots h_1$, where each $h_j$ is conjugate either to $g^n$ or to its inverse:
\[ 
\forall\,  1\leq j\leq m, \; \exists\,  \psi_j\in G, \; h_j=\psi_j g^n \psi_j^{-1} \text{ or } \psi_j g^{-n} \psi_j^{-1}.
\]
Let $x_0\in X$ be a base point. 
To such a choice of a base point and of a presentation of $h$, we associate three sequences  $(a_i)$, $(b_i)$ and $(x_i)$, $1\leq i\leq m$, by setting:
$a_i$ is the projection of $x_{i-1}$ on $\axe(h_i)$, $b_i = h_i (a_i)$ and $x_{i} = h_i (x_{i-1})$.
We say that $h_m \cdots h_1$ is an \emph{admissible presentation} of $h$ (with respect to the base point $x_0$) if all the segments $[x_{i-1},a_i]$ are neutral (hence also the segments $[b_{i},x_i]$).
\end{definition}

\begin{lemma}[{\cite[Lemma 2.13]{CL}}]
\label{lem:presentation}
Any element $h\in \llb g^n \rrb$ admits at least one admissible presentation.
\end{lemma}

\begin{proof}
Let $ h_m \cdots h_1$ be a presentation of $h$, and $\mathcal{I}$ the set of indexes $1 \le i \le m$ such that $[x_{i-1},a_i]$ is not neutral. 
Assume $\mathcal{I}$ not empty (otherwise the presentation is already admissible), and consider $i \in \mathcal{I}$.
By definition there exists $f$ a conjugate of $g^n$ or $g^{-n}$ such that $[y,z] = \axe(f) \cap [x_{i-1},a_i]$ is a relator of size $>\frac{1}{2}\ell$, with $y \in [x_{i-1},z]$. 
Up to replacing $f$ by $f^{-1}$, we can assume $z \in [y, f(y)]$.
 
We rewrite $h_i$ as the product of three conjugates of $g^n$:
$$h_i = (h_i f^{-1}h_i^{-1}) h_i f.$$
This yields a new presentation of $h$, and so also new sequences of points. 
In the sequence $(x_i)$ we have two new points, $x'_i = f(x_{i-1})$ and $x'_{i+1} = h_i(x'_i)$.
The point $x'_{i+2} = h_i f^{-1} h_i^{-1}(x'_{i+1})$ is equal to the old point $x_i = h_i(x_{i-1})$.
So the non-neutral segment $[x_{i-1},a_i]$ is replaced by three new segments $[x_{i-1},y]$, $[x'_i,a_{i}]$ et $[x'_{i+1}, h_i(f(y))]$ (see Figure \ref{fig:admissible}). 
Other segments $[x_{j-1},a_j]$ are left unchanged.
Since $[y,z]$ is a relator of size $>\frac{1}{2}\ell$, we have $\dist(z, f(y)) <\dist(y,z)$, and we get the inequalities:
\begin{align*}
\dist(x_{i-1},y) = \dist(x'_{i+1},h_i(f(y))) < \dist(x_{i-1},a_i) - \tfrac{1}{2}\ell; 
&&
\dist(x'_{i},a_i) < \dist(x_{i-1},a_i).
\end{align*}

\begin{figure}
\hspace{2cm}
\begin{tikzpicture}[scale = 1.7,font=\small,thick]
\coordinate [label=above:$a_i$] (ai) at (0,0);
\coordinate [label=above:$b_i$] (bi) at (1.5,0);
\coordinate [label=left:$x_{i-1}$](xi-1) at (-1,-2) {};
\coordinate [label=right:{$x_i = x_{i+2}' = h_i(x_{i-1})$}](xi) at (2.5,-2) {};
\coordinate [label=left:$y$](y) at ($ (xi-1)!1/3!(ai) $) {};
\coordinate [label=left:$z$](z) at ($ (xi-1)!2/3!(ai) $) {};
\coordinate [label=left:$h_i(y)$] (hy) at ($ (xi)!1/3!(bi) $) {};
\coordinate [label=left:$h_i(z)$] (hz) at ($ (xi)!2/3!(bi) $) {};
\coordinate [label=below:{$x_i'=f(x_{i-1})$}, right= of y,xshift=40] (xi') {};
\coordinate [label=above right:{$x_{i+1}'=h_i(x_i')$}, right= of hz,xshift=40] (xi+1') {};
\coordinate [right=of xi-1,xshift=10] (axef_d) {};
\coordinate [right=of z,xshift=30] (axef_f) {};
\coordinate [above left=of ai,xshift=-20] (axeh_d) {};
\coordinate [above right=of bi,xshift=20] (axeh_f) {};
\coordinate [above=of xi+1',yshift=0] (axe_d) {};
\coordinate [below=of xi+1',yshift=-20] (axe_f) {};
\coordinate [label=above:$f(y)$](fy) at ($ (z)!1/2!(axef_f) $) {};
\coordinate [label=above left:$h_i(f(y))$](hfy) at ($ (hz)!1/2!(axe_d) $) {};
\draw[draw=lightgray] (xi-1)--(y) (z)--(ai) (fy)--(xi') (xi)--(bi) (hfy)--(xi+1');
\draw[draw=black,->] (axef_d)--(y)--(z)--node[pos=1,auto]{$\axe(f)$}(axef_f);
\draw[draw=black,->] (axeh_d)--(ai)--(bi)--node[pos=.95,auto]{$\axe(h_i)$}(axeh_f);
\draw[draw=black,->] (axe_d)--(hz)--(hy)--node[pos=.9,auto]{$\axe(h_i f^{-1} h_i^{-1})$}(axe_f);
\foreach \v in {ai,bi,xi,xi-1,y,z,fy,hy,hz,xi',xi+1',hfy} {
\node[bullet] at (\v) {};
}
\end{tikzpicture}
\caption{Proof of Lemma \ref{lem:presentation}.}
\label{fig:admissible}
\end{figure}
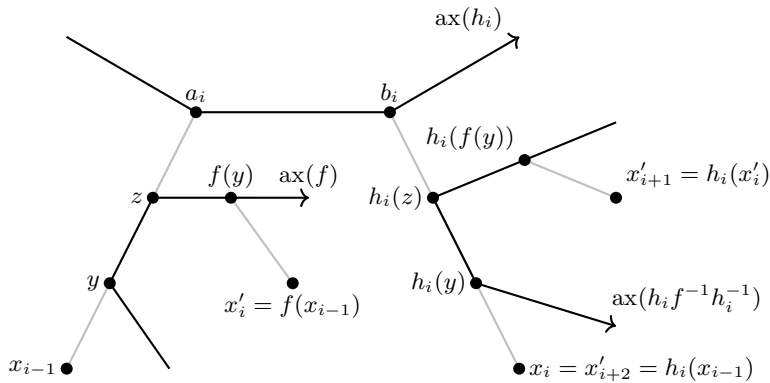

Since this modification does affect other non-neutral subsegment $[x_{j-1},a_j]$, we can simultaneously perform this modification for all $i \in \mathcal{I}$. 
We obtain a new presentation of $h$, a new list of bad indexes $\mathcal{I}'$.
If this list is not empty, the maximum of the lengths $\dist(x_{j-1},a_j)$ taken over all non-neutral segments $[x_{j-1},a_j]$, $j \in \mathcal{I}'$, has dropped by at least $1$ since we are working with a simplicial tree.
By induction, after finitely many such steps we obtain an admissible presentation for $h$.
\end{proof}

\begin{lemma}[{\cite[Lemma 2.15]{CL}}]
\label{lem:mini}
Let $h = h_m \cdots h_1$ be an admissible presentation with base point $x_0$. 
If there exist two indexes $j > i$ such that $h_j = h_i^{-1}$, then $h$ admits an admissible presentation with the same base point, and with only $m-2$ factors.
\end{lemma}

\begin{proof}
We assume $j\geq i+2$, otherwise the simplification is obvious. 
Then by writing  
\begin{align*}
h &= h_m \cdots h_{j+1} h_i^{-1} h_{j-1} \cdots h_{i+1} h_i h_{i-1} \cdots h_1 \\
&= h_m \cdots  h_{j+1} (h_i^{-1} h_{j-1} h_i) \cdots (h_i^{-1} h_{i+1} h_i)  h_{i-1} \cdots h_1
\end{align*}
we get the expected admissible presentation with $m-2$ factors.
\end{proof}

\subsection{A variant of Greendlinger's Lemma}
\label{sec:Greendlinger}

Now, we start the proof of Theorem \ref{thm:SCT}. 

Consider $B>0$ the constant given by the tightness of $g$ (see Definition \ref{def:tight}), and choose $n \ge 1$ such that the translation length $\ell  = n\ell(g)$ of $g^n$ satisfies 
\[
\l > 12B.
\]

\begin{remark}
The constant $\frac{1}{12}$ should be compared with the constant $\frac{1}{6}$ of the condition $C'(1/6)$ mentioned in the introduction.
One could increase a little this constant ($\frac{1}{8}$ seems plausible) if we relaxed assertion \ref{SCT2}, for instance by asking that $\ell(h) > \frac{1}{2}\ell(g^n)$.
However it seems that by nature this method of proof cannot provide any sharp bound on the power necessary to get a proper normal subgroup, so it is not clear what would be the point of trying to replace our $\frac{1}{12}$ by a slightly better constant.  
\end{remark}
Let $h$ be a non trivial element in $\llb g^n \rrb$, and $x_0 \in \axe(h)$ (or a priori in $\Fix(h)$ if $h$ is elliptic, but the proof will show that $h$ is always loxodromic).
By Lemma \ref{lem:presentation}, there exists an admissible presentation 
\[
h=h_m \cdots h_1
\]
with respect to the point $x_0$.
We assume that $m$ is minimal, among all admissible presentations of $h$ with base point $x_0$. Let $(a_i)$, $(b_i)$ and $(x_i)$ be the sequences from Definition~\ref{def:admissible}.

\begin{definition}[Configuration of order $k$, {\cite[\S 2.5]{CL}}]
A sequence of points $(c_{-1}$, $c_0$, $\cdots$, $c_k, c_{k+1})$ in $[x_0,x_j]$ is a   \emph{configuration of order $k \ge 1$} for the segment $[x_0,x_j]$ if
\begin{enumerate}[(i)]
\item \label{point:i} The sequence is monotonous, with $x_0 = c_{-1}$ and $x_j = c_{k+1}$;
\item \label{point:ii} For all $0 \leq i \leq k$, the segment $[c_i,c_{i+1}]$ is either neutral or a relator, with the following rules:
        \begin{enumerate}[(a)]
         \item Two consecutive segments are not both neutral;
         \item The last segment $[c_k,c_{k+1}]= [c_k, x_j]$ is neutral;
         \item The second segment $[c_0,c_1]$ is a relator of size $\geq \l-2B$ if $[c_1, c_2]$ is neutral, or of size $\geq \l-3B$ otherwise;
         \item For any other relator $[c_{i-1}, c_i]$, with $i>1$, the size is $>4B$ if $[c_{i}, c_{i+1}]$ is neutral and $>3B$ otherwise.
        \end{enumerate}
\item \label{point:iii} For all $0 \leq i \leq k$, if $[c_i,c_{i+1}]$ is a relator, then there exists an index $l_i$ with $1\leq l_i\leq j$ such that $h_{l_i}$ is the support of the relator $[c_i,c_{i+1}]$.
\end{enumerate}
Observe that properties \ref{point:ii} and \ref{point:iii} do not concern the initial segment $[x_0,c_0]$.
\end{definition}

\begin{lemma}[Greendlinger's Lemma, {\cite[Lemma 2.16]{CL}}]
\label{lem:infernal}
For each $j = 1,\cdots,m$, there exists $k \ge 1$ such that the segment $[x_0,x_j]$ admits a configuration of order $k$.

Moreover if $j \ge 2$ and $k = 1$, then the initial segment $[x_0,c_0]$ of the configuration has diameter $> 2B$.
\end{lemma}

\begin{proof}
We proceed by induction on $j$.
When $j = 1$, by setting $c_0 = a_1, c_1 = b_1$ we obtain a configuration of order 1. 

Assume now that $[x_0,x_{j}]$ admits a configuration $(c_i)_{-1\leq i\leq k+1}$ of order $k$. We want to construct a configuration of order $k'$ for $[x_0,x_{j+1}]$. We work inside the tripod $T \subset X$ defined by $x_0,x_j,x_{j+1}$. We denote by $p$ the branch point of $T$.

The admissible presentation of $h$ being minimal, Lemma \ref{lem:mini} implies that $h_{i_1} \neq h_{i_2}^{-1}$ for all $m \ge i_2 > i_1 \ge 1$. 
Consequently any relator $[c_i,c_{i+1}]$ is supported by some $h_{l_i}$ which is not $h_{j+1}^{-1}$, and $[c_i,c_{i+1}]\cap\axe(h_{j+1})$ has length at most $B$.
The following inequalities follow immediately from this and from the fact that $[a_{j+1},b_{j+1}]$ is a relator of size $\geq\l$ and $[x_j,a_{j+1}]$, $[b_{j+1},x_{j+1}]$ are neutral:

\begin{fact}\label{fact:easy}
Let $ 0 \le i \le k$.
\begin{itemize}
\item If $[c_i, c_{i+1}]$ is neutral then
\begin{equation} \label{fact:easy1}
\diam([a_{j+1},b_{j+1}] \cap [c_i,c_{i+1}]) \le \tfrac{1}{2}\ell;
\end{equation}
\item If $[c_i, c_{i+1}]$ is a relator then
\begin{align}
\diam([a_{j+1},b_{j+1}] \cap [c_i,c_{i+1}]) &\le B; \label{fact:easy2}  \\  
\diam([x_j,a_{j+1}] \cap [c_i,c_{i+1}]) &\leq  \tfrac{1}{2}\ell; \label{fact:easy3}\\
\diam([b_{j+1},x_{j+1}] \cap [c_i,c_{i+1}]) &\leq  \tfrac{1}{2}\ell. \label{fact:easy4}
\end{align}
\end{itemize}
\end{fact}

\begin{figure}[h]
\[
\begin{tikzpicture}[scale = 1.7,font=\small,thick]
\coordinate [label=above right:$p$] (p) at (0,0);
\coordinate [label=above:$x_j$] (xj) at (0:1);
\coordinate [label=above right:$x_0$] (x0) at (135:1.2);
\coordinate [label=above left:$x_{j+1}$] (xj+1) at (-135:1.2);
\coordinate [label=above right:$c_0$](c0) at ($ (x0)!1/2!(p) $) {};
\coordinate [label=above left:$b_{j+1}$](b) at ($ (p)!2/3!(xj+1) $) {};
\foreach \v in {p,xj,xj+1,x0,c0,b} {
\node[bullet] at (\v) {};
}
\draw (x0)--(c0) to["$>2B$"] (p)--(xj) (xj+1)--(b) to["$>\frac{1}{3}\ell$", swap] (p);
\end{tikzpicture}
\]
\caption{} \label{fig:c0 and bj+1}
\end{figure}

\begin{fact}[see Figure \ref{fig:c0 and bj+1}]\label{fact:c0 and bj+1}~
\begin{enumerate}
\item \label{fact:c0}
$c_0 \in [x_0, p[$, and $\diam [c_0,p] > 2B$;
\item \label{fact:bj+1}
$b_{j+1} \in [x_{j+1}, p[$, and $\diam [b_{j+1},p] >\frac{1}{3}\l$.
\end{enumerate}
\end{fact}

\begin{proof}
\begin{enumerate}
\item
The relator $[c_0, c_1]$ has size $\geq \l - 3B > 9B$, and its intersection with the relator $[a_{j+1}, b_{j+1}]$ of size $\l$ has diameter at most $B$.
So $[c_0, c_1]$ is not contained in $[x_j,x_{j+1}]$, it intersects at most one of the neutral segments $[x_{j+1}, b_{j+1}]$ and $[a_{j+1}, x_j]$ and we obtain that
\[\diam [c_0,p] \ge (\l - 3B) - B - \tfrac12 \l = \tfrac12 \l - 4B > 2B.\]
\item
Assume that $\diam [b_{j+1},a_{j+1}] \cap [p,x_j] > \frac12 \l$, otherwise the conclusion is direct.
By Fact \ref{fact:easy}\eqref{fact:easy1}, the previous point and the assumption that there are no consecutive neutral segments, $[b_{j+1},a_{j+1}]$ intersects at least one relator $[c_i,c_{i+1}]$.
Fact \ref{fact:easy}\eqref{fact:easy2} together with the assumption that each relator of the configuration has size $> B$, imply that $[c_i,c_{i+1}]$ is not contained in $[a_{j+1},b_{j+1}]$, so it must intersect one of the ends of $[b_{j+1},a_{j+1}] \cap [p,x_j]$.
So $[b_{j+1},a_{j+1}] \cap [p,x_j]$ intersects at most two relators and one neutral segment; we obtain that
\[
\diam [b_{j+1},p] > \l - 2B - \tfrac12 \l = \tfrac12 \l -2B > \tfrac12 \l - \tfrac16 \l = \tfrac13 \l. \qedhere
\]
\end{enumerate}
\end{proof}

In particular the tripod $T$ is not degenerate at $x_0$ or at $x_{j+1}$.
To conclude, we distinguish three cases according to the position of $a_{j+1}$ with respect to the branch point $p$.\\

\noindent\emph{First case}: $a_{j+1} \in [x_{j+1}, p]$, see Figure \ref{fig:2cas}(a).

We set $c_0 = a_{j+1}$ and $c_1 = b_{j+1}$, and we obtain a configuration of order 1 for $[x_0,x_{j+1}]$.
This case includes the degenerate situation where $x_j$ is the branch point of $T$.
So Lemma \ref{lem:infernal} is proved in this case, where Fact \ref{fact:c0 and bj+1} gives the second assertion of the lemma.

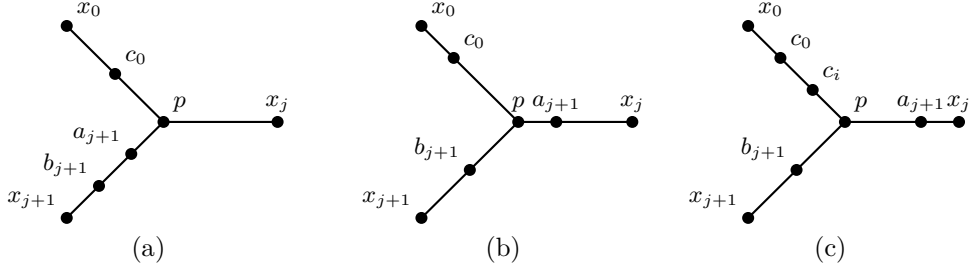
\begin{figure}[h]
\begin{align*}
\begin{array}{c}
\begin{tikzpicture}[scale = 1.5,font=\small,thick]
\coordinate [label=above right:$p$] (p) at (0,0);
\coordinate [label=above:$x_j$] (xj) at (0:1);
\coordinate [label=above right:$x_0$] (x0) at (135:1.2);
\coordinate [label=above left:$x_{j+1}$] (xj+1) at (-135:1.2);
\coordinate [label=above right:$c_0$](c0) at ($ (x0)!1/2!(p) $) {};
\coordinate [label=above left:$a_{j+1}$](a) at ($ (p)!1/3!(xj+1) $) {};
\coordinate [label=above left:$b_{j+1}$](b) at ($ (p)!2/3!(xj+1) $) {};
\foreach \v in {p,xj,xj+1,x0,c0,a,b} {
\node[bullet] at (\v) {};
}
\draw (x0)--(p)--(xj) (p)--(xj+1);
\end{tikzpicture}
\\
\text{(a)}
\end{array}
&&
\begin{array}{c}
\begin{tikzpicture}[scale = 1.5,font={\fontsize{9}{9}\selectfont},thick]
\coordinate [label=above:$p$] (p) at (0,0);
\coordinate [label=above:$x_j$] (xj) at (0:1);
\coordinate [label=above right:$x_0$] (x0) at (135:1.2);
\coordinate [label=above left:$x_{j+1}$] (xj+1) at (-135:1.2);
\coordinate [label=above right:$c_0$](c0) at ($ (x0)!1/3!(p) $) {};
\coordinate [label=above:$a_{j+1}$](a) at ($ (p)!1/3!(xj) $) {};
\coordinate [label=above left:$b_{j+1}$](b) at ($ (p)!1/2!(xj+1) $) {};
\foreach \v in {p,xj,xj+1,c0,x0,a,b} {
\node[bullet] at (\v) {};
}
\draw (x0)--(p)--(xj) (p)--(xj+1);
\end{tikzpicture}
\\
\text{(b)}
\end{array}
\begin{array}{c}
\begin{tikzpicture}[scale = 1.5,font={\fontsize{9}{9}\selectfont},thick]
\coordinate [label=above right:$p$] (p) at (0,0);
\coordinate [label=above:$x_j$] (xj) at (0:1);
\coordinate [label=above right:$x_0$] (x0) at (135:1.2);
\coordinate [label=above left:$x_{j+1}$] (xj+1) at (-135:1.2);
\coordinate [label=above right:$c_0$](c0) at ($ (x0)!1/3!(p) $) {};
\coordinate [label=above right:$c_i$](ci) at ($ (x0)!2/3!(p) $) {};
\coordinate [label=above:$a_{j+1}$](a) at ($ (p)!2/3!(xj) $) {};
\coordinate [label=above left:$b_{j+1}$](b) at ($ (p)!1/2!(xj+1) $) {};
\foreach \v in {p,xj,xj+1,c0,ci,x0,a,b} {
\node[bullet] at (\v) {};
}
\draw (x0)--(p)--(xj) (p)--(xj+1);
\end{tikzpicture}
\\ 
\text{(c)}
\end{array}
\end{align*}
\caption{The cases in the proof of Lemma \ref{lem:infernal}.}
\label{fig:2cas}
\end{figure}

\noindent\emph{Second case}: $a_{j+1} \in ]p,x_j]$, with $\dist(p, a_{j+1}) \le 2B$. See Figure \ref{fig:2cas}(b).

We set $c_0=p$ and $c_1=b_{j+1}$ and we obtain a configuration of order 1 for $[x_0,x_{j+1}]$, and again Fact \ref{fact:c0 and bj+1} achieves the proof of the lemma in this case.
\medskip

\noindent\emph{Third case}: $a_{j+1} \in ]p,x_j]$,  with $\dist(p, a_{j+1}) > 2B$. See Figure \ref{fig:2cas}(c).

There is a unique index $i \ge 0$ such that $p \in ]c_i,c_{i+1}]$.
Consider two subcases depending if $i=0$ or not.

First if $i=0$, then by Fact \ref{fact:easy}(\ref{fact:easy2}) the segment $[p,c_{1}]$ has length at most $B$, and $[c_1, c_2]$ is neutral (because the intersection $[c_1,c_2] \cap [b_{j+1}, a_{j+1}]$ has diameter $>B$), so $[c_0,p]$ is a relator of size $\geq \l-3B$.
Moreover by Fact \ref{fact:c0 and bj+1}\ref{fact:bj+1} the segment $[p,b_{j+1}]$ is a relator of size $> \frac{1}{3}\l>4B$.
We keep $c_0$ and set $c_1=p$ and $c_2=b_{j+1}$. This gives us a configuration of order $2$ for $[x_0,x_{j+1}]$.

If $i \geq 1$ we consider again two subcases according to the nature of the segment $[c_i,c_{i+1}]$:

If $[c_i,c_{i+1}]$ is neutral then $[c_i,p]$ also is.

If $[c_i,c_{i+1}]$ is a relator then by Fact \ref{fact:easy}(\ref{fact:easy2}) the segment $[p,c_{i+1}]$ has length at most $B$, and $[c_{i+1},c_{i+2}]$ is neutral because its intersection with $[b_{j+1}, a_{j+1}]$ has diameter larger than $B$. 
Thus the relator $[c_i,c_{i+1}]$ has size $>4B$, and $[c_i,p]$ is still a relator, of size $>3B$.

Moreover by Fact \ref{fact:c0 and bj+1}\ref{fact:bj+1} the segment $[p,b_{j+1}]$ is a relator of size $> \frac{1}{3}\l>4B$.
So in both subcases, by keeping the $c_j$ with $j \le i$, and setting $c_{i+1} = p$, $c_{i+2}= b_{j+1}$, we obtain a configuration of order $i+2$ for $[x_0,x_{j+1}]$.
\end{proof}

\begin{proof}[Proof of Theorem \ref{thm:SCT}\ref{SCT2}]Recall that $\dist(x_0,x_m) = \ell(h)$.
By Lemma \ref{lem:infernal} there exists $(c_i)$ a configuration of order $k$ for $[x_0,x_{m}]$.

If $k \ge 2$, we have at least two distincts relators: $[c_0,c_1]$ of size $\geq\l-3B$, and $[c_{k-1},c_k]$ of size $>3B$. 
We conclude that $\ell(h) > \ell$. 

If $k = 1$, either $h$ is conjugate to $g^n$, or by the second assertion of Lemma \ref{lem:infernal} we have $\dist(x_0,c_0)> 2B$. 
Moreover $\dist(c_0,c_1)\geq \l-2B$, so that again we get $\ell(h)> \ell$.
\end{proof}

\section{Polynomial automorphisms}

\subsection{The almagamated product structure}

Let $\K$ be any field.
We denote by $\Aut(\K^2)$ the group of polynomial automorphisms of the affine plane $\K^2$.
Let $A = \GL_2(\K) \ltimes \K^2$ be the subgroup of affine automorphisms, and $B$ the subgroup of elementary automorphisms: 
\[
B = \{(x,y) \mapsto (ax + P(y), by +c);\; a,b, c\in \K, ab \neq 0, P\in \K[y] \}.
\]

\begin{theorem}[Jung--van der Kulk, see for instance \cite{LJung}] 
The group $\Aut(\K^2)$ is the amalgamated product of its subgroups $A$ and $B$ along their  intersection.  
\end{theorem}

The group $\Aut(\K^2)$ acts on its associated Bass-Serre $T$.
Vertices of $T$ are the left cosets $f A$ and $f B$, $f \in \Aut(\K^2)$, and edges are left cosets $f (A \cap B)$.
Observe that this tree is not locally finite, even when working over a finite field.
Indeed the edges issued from the vertex $\id B$ are parametrized by the left cosets $B / A$, which we can represent by the automorphisms $(x+y^2P(y),y)$, $P \in \K[y]$.

The group $\Aut(\K^2)$ always contains non-abelian free groups:

\begin{lemma}
Let $b : (x,y) \mapsto (-x + y^2, y)$, 
 $a_\lambda \colon (x,y) \mapsto (\lambda x + y,x)$ where $\lambda \in \K$, $a_\infty = \id$, and set $g_\lambda = a_\lambda b a_\lambda^{-1}$.
Then the subgroup generated by the involutions $g_\lambda$ is a free product of $\Z/2$ parametrized by  $\P^1_\K = \K \cup \{\infty\}$ $($which has cardinal at least $3)$, in particular it contains a copy of the free group $\Z * \Z$.
\end{lemma}

\begin{proof}
By construction $g_\lambda$ fixes the vertex $a_\lambda B$ but not the vertex $\id A$.
Denote $U_\lambda \subset T$ the subtree of points whose projection on the segment $[\id A, a_\lambda B]$ is equal to $a_\lambda B$.
Then for every $\lambda' \neq \lambda$ we have  $g_\lambda U_{\lambda'} \subset U_\lambda$.
Given a reduced word $g_{\lambda_n} \dots g_{\lambda_1}$, choose $\lambda$ distinct from both $\lambda_1$ and $\lambda_n$, and observe that $g_{\lambda_n} \dots g_{\lambda_1} U_{\lambda} \subset U_{\lambda_n}$ so that this word is not trivial.
\end{proof}

\begin{remark}
The argument in the proof is standard and often called the ``ping-pong lemma'', see for instance \cite[II.B.24]{delaHarpe}.
Observe also that the fact that the index set has cardinality at least 3 is important when dealing with involutions.
\end{remark}
 
\subsection{WPD elements}
In \cite{Lamy}, the elliptic elements in $\Aut(\C^2)$ where classified, with in particular the following characterization of elliptic elements admitting a fixed subtree of large diameter:

\begin{proposition}[{ \cite[Proposition 3.3]{Lamy}}]\label{prop_lamy_stab}
With respect to the action of $\Aut(\C^2)$ on its Bass--Serre tree, the stabilizer of any path of length at least 7 is finite,  and more precisely is conjugate to a cyclic group of maps of the form $(x,y) \mapsto (\alpha x, \beta y)$, with $\alpha, \beta$ primitive roots of the same order.
\end{proposition}

 In fact the same proof would apply to any field $\K$ of characteristic zero.
From this we can deduce that over a field of characteristic zero, any loxodromic element in $\Aut(\K^2)$ is WPD.
In this note we will content ourselves by giving a simple proof of the existence of WPD elements, that is valable in any characteristic (up to a twist in characteristic 2). In fact, computations of Proposition \ref{pro:auto WPD} are particular cases of the ones used in the proof of Proposition \ref{prop_lamy_stab}, and so give the flavour of the full proof. 

We will work with the following involutions 
\begin{align*}
b = (-x+y^2, y) \in B \smallsetminus A, 
&&
t = a_0 = (y,x) \in A \smallsetminus B.
\end{align*}

\begin{proposition}[Compare with {\cite[Lemma 4.23]{MO}}]
\label{pro:auto WPD}
Assume $\car \K \neq 2$.
Then $bt = (x^2-y, x) \in \Aut(\K^2)$ is loxodromic and satisfies the WPD property.
\end{proposition}

\begin{proof}
We claim that there are only finitely many automorphisms fixing pointwise the following path of length 6 inside the axis of $b t$.

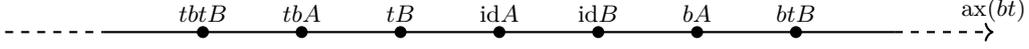
\[
\begin{tikzpicture}[node distance=1.3cm,font=\small,thick]
\coordinate[label=above:$t bt B$] (tbtB) at (0,0);
\coordinate [right=of tbtB, label=above:$t b A$] (tbA);
\coordinate [right=of tbA, label=above:$t B$] (tB);
\coordinate [right=of tB, label=above:$\id A$] (A);
\coordinate [right=of A, label=above:$\id B$] (B);
\coordinate [right=of B, label=above:$bA$] (bA);
\coordinate [right=of bA, label=above:$bt B$] (btB);
\coordinate [left=of tbtB] (tail);
\coordinate [left=of tail] (far tail);
\coordinate [right=of btB] (head);
\coordinate [right=of head, label=above:$\axe(bt)$] (far head);
\draw (tail)--(head);
\draw[dashed] (far tail)--(tail);
\draw[dashed,->] (head)--(far head);
\foreach \v in {tbtB,tbA,tB,A,B,bA,btB} {
  \node[bullet] at (\v) {};
}
\end{tikzpicture}
\]

By definition of the Bass-Serre tree, any $f$ stabilizing the edge between $\id A$ and $\id B$ is an element of $A \cap B$, hence has the form
$$f = (\alpha x + \beta y + \gamma, \delta y + \eps),$$
where $\alpha\delta \neq 0$ since $f$ is invertible.
Now such an $f$ fixes the vertex $t B$ if and only if 
\[t f t^{-1} = (\delta x + \eps, \beta x + \alpha y + \gamma) \in B.\]
So $f$ fixes $t B$ if and only if $\beta = 0$, which we now assume.

Similarly, $f$ fixes the vertices $bA$ and $bt B$ if and only if $b f b^{-1}$ is an element in $A$ of the form $(\alpha'x + \gamma', \delta'y +\eps')$.
We have:
\begin{align*}
b f b^{-1} &= (-x+y^2, y) \circ (\alpha x + \gamma, \delta y + \eps) \circ (-x+y^2,y) \\
&=  (-x+y^2, y) \circ (-\alpha x + \alpha y^2 + \gamma, \delta y + \eps) \\
&= (\alpha x - \alpha y^2 - \gamma + (\delta y + \eps)^2, \delta y + \eps).
\end{align*}
So $f$ fixing the vertices $bA$ and $bt B$ implies $\delta^2 = \alpha$ and $2\delta\eps = 0$.
Since $\car \K \neq 2$, this gives $\eps = 0$.

By symmetry of the argument, the only automorphisms that fix the path of length $6$ from $t b t B$ to $t b A$ are the $(\alpha x, \delta y)$ with $\alpha^2 = \delta$ and $\delta^2 = \alpha$, which implies that $\alpha, \delta$ are cubic roots of the unity. 
Finally the pointwise stabilizer of this path is a finite group of order at most $3$, and we can apply Lemma \ref{lem:cor4.2}.
\end{proof}

\begin{remark}
When $\car \K = 2$, we can work with the elementary involution  $b = (x+y^3, y)$, and show that  $bt = (x^3+y,x) \in \Aut(\K^2)$ is a loxodromic WPD isometry, with essentially the same proof.
\end{remark}

\begin{remark}
Over an infinite field of characteristic $p >0$, the loxodromic map $g = (x^p-y,x)$ is not
WPD.
Indeed, the group of translations $T = \{(x + a, y + b) \mid a,b \in \K\}$ is normalized by $g$:
\[g \circ (x+a, y+ b) \circ g^{-1} = (x+a^p-b, y + a) \in T.\]
So for any $t \in T$, and for any $n \in \Z$, there exists $t_n \in T$ such that $g^n \circ t_n = t \circ g^n$.
The vertex $\id A$ is in the axis of $g$, and is fixed by $T$, so we get 
\[
g^n \id A = g^n t_n A = t g^n \id A.
\] 
In consequence the infinite group $T$ fixes pointwise the axis of $g$.
\end{remark}

\begin{definition}[Acylindricity, {\cite[Definition 5.30]{DGO}}]
Let $G \action X$ be a group acting on a metric space.
We say that the action of $G$ is \textit{acylindrical} if for all $d$ there
exist $R_d > 0$, $N_d > 0$ such that for all $x,y \in X$ with $\dist(x,y) > R_d$
the set
\[
\{ g \in G \mid \dist(x,gx) \le d, \dist(y,gy) \le d \}
\]
contains at most $N_d$ elements.
\end{definition}

\begin{proposition}
Let $\K$ be a field of characteristic zero.
Then the action of $\Aut(\K^2)$ on its Bass--Serre tree is acylindrical if and only if $\K$ contains only finitely many roots of the unity.
\end{proposition}

\begin{proof}
Let $\alpha$ be a primitive $n$th root of unity, and set $g(x,y) = (x^{n+1} -y,x)$, $f(x,y) = (\alpha x, \alpha y)$.
Then $f$ and $g$ commute, so the elliptic isometry $f$ fixes pointwise the axis of $g$.
So if $\K$ contains primitive $n$th roots of the unity for arbitrary large $n$, we obtain vertices at an arbitrary large distance with stabilizer of cardinal at least $n$.  

The converse statement follows directly from Proposition \ref{prop_lamy_stab}.
\end{proof}

\begin{example}[WPD does not imply tight]
\label{ex:WPDnotTight}
Let $\K$ be a field of characteristic $\neq 2$ and containing $j$, a primitive third root
of unity. 
Let $g = (x^2-y,  x)$, and $f = (jx,j^2y)$.
We have
\[
g \circ f = (j^2(x^2-y), j x) = f^2 \circ g.
\]
It follows that $f$ commutes with $g^2$ but not with $g$.
In particular $\axe(fgf^{-1}) = \axe(g)$ but $fgf^{-1}$ is not equal to $g$ or
$g^{-1}$.
Hence $g$ is not tight, but $g$ is WPD by Proposition \ref{pro:auto WPD}.
\end{example}

\begin{example}[tight does not imply WPD]
\label{ex:TightNotWPD}
Let $g \in G$ be a tight loxodromic element for an action $G \action
X$.
Then extend the action to $G \times \Z$, by letting the $\Z$ factor act
trivially.
Then $g$ is still tight with respect to this action, but it cannot be WPD
because the centralizer of $g$ contains the $\Z^2$ generated by $(g,0)$ and
$(1_G,1)$.
\end{example}

\bibliographystyle{myalpha}
\bibliography{biblio}

\end{document}